\documentclass[12pt]{article}
\usepackage[margin=2.4cm]{geometry}

\usepackage{graphicx}
\graphicspath{{./images}}
\usepackage{amsmath,amsthm,latexsym,amsfonts,amssymb,mathrsfs}
\usepackage[usenames]{color}
\usepackage{tikz}
\usepackage{comment}
\usepackage{subcaption}
\usetikzlibrary{math}
\usepackage{multirow}
\usepackage{makecell}
\usepackage{amsmath} 
\usepackage{graphicx} 
\usepackage{booktabs}

\usepackage{cancel,soul,ulem}

\usepackage{url}

\newcommand{\calL}{\mathcal{L}}

\newcommand{\va}{{\bf a}}
\newcommand{\vb}{{\bf b}}
\newcommand{\vc}{{\bf c}}
\newcommand{\vd}{{\bf d}}

\newcommand{\vf}{{\bf f}}

\newcommand{\vu}{{\bf u}}
\newcommand{\vv}{{\bf v}}
\newcommand{\vw}{{\bf w}}
\newcommand{\vz}{{\bf z}}

\newcommand{\vzero}{{\bf 0}}
\newcommand{\calE}{{\mathcal  E}}
\newcommand{\calF}{{\mathcal  F}}

\newcommand{\bsd}{{\boldsymbol{d}}}
\newcommand{\bsx}{{\boldsymbol{x}}}
\newcommand{\bsy}{{\boldsymbol{y}}}

\newcommand{\tth}{{\boldsymbol{\theta}}}

\newcommand{\vg}{{\bf{g}}}

\newcommand{\vsigma}{\boldsymbol{\sigma}}

\newcommand{\R}{\mathbb{R}}

\numberwithin{equation}{section}

\newcommand{\veps}{{\boldsymbol{\varepsilon}}}
\newcommand{\ip}[1]{\langle#1 \rangle}

\newtheorem{theorem}{Theorem}[section]

\title{A decomposition-based robust training of physics-informed neural networks for nearly incompressible linear elasticity}
\author{J. Dick\thanks{School of Mathematics and Statistics, University of New South Wales, Sydney, Australia. \tt{}}, ~S. Ko\thanks{Department of Mathematics, Inha University, Incheon, Republic of Korea. \tt{}},
~Q. T. Le Gia\footnotemark[1], ~K. Mustapha\footnotemark[1]
~and
~S. Park\footnotemark[2]
}

\date{}
\begin{document}
\maketitle
\begin{abstract}
Due to divergence instability, the accuracy of low-order conforming finite element methods for nearly incompressible elasticity equations deteriorates as the Lam\'e coefficient  $\lambda\to\infty$, or equivalently as the Poisson ratio $\nu\to1/2$. This phenomenon, known as locking or non-robustness, remains not fully understood despite extensive investigation. In this work, we illustrate first that an analogous instability arises when applying the popular Physics-Informed Neural Networks (PINNs) to nearly incompressible elasticity problems, leading to significant loss of accuracy and convergence difficulties. Then, to overcome this challenge, we propose a robust decomposition-based PINN framework that reformulates the elasticity equations into balanced subsystems, thereby eliminating the ill-conditioning that causes locking. Our approach simultaneously solves the forward and inverse problems to recover both the decomposed field variables and the associated external conditions. We will also perform a convergence analysis to further enhance the reliability of the proposed approach. Moreover, through various numerical experiments, including constant, variable and parametric Lam\'e coefficients, we illustrate the efficiency of the proposed methodology.
\end{abstract}

\section{Introduction}

Consider the linear elasticity problem on a bounded, polygonal domain \(\Omega \subset \mathbb{R}^d\) for \(d \in \{2,3\}\). We seek the displacement field $\vu(\boldsymbol{x})$ satisfying
\begin{equation}\label{eq:EE}
-\nabla\cdot\vsigma\bigl(\vu(\bsx)\bigr) \,=\, \vf(\bsx),
\quad \bsx \in \Omega,
\end{equation}
subject to the following Dirichlet boundary conditions:
$\vu(\bsx) = \vg(\bsx)$ for $\bsx \in \Gamma := \partial\Omega.$
Here, \(\vsigma(\vu) \in [L^2(\Omega)]^{d\times d}\) is the Cauchy stress tensor, defined by
\begin{equation}\label{eq:sig}
\vsigma(\vu) = \lambda\,\bigl(\nabla\cdot \vu\bigr)\,{ \bf I} + 2\mu\,\veps(\vu),
\quad \bsx \in \Omega,
\end{equation}
where the symmetric strain tensor is given by
\[
\veps(\vu) = \tfrac12\bigl(\nabla \vu + (\nabla \vu)^T\bigr).
\]

In this formulation, \(\lambda\) and \(\mu\) are the Lam\'e coefficients, which can be constants, functions of $\bsx\in \Omega$, or random fields, \(\vf: \Omega \to \mathbb{R}^d\) is the prescribed body force per unit volume, \(\vg: \Gamma \to \mathbb{R}^d\) is the given boundary condition,  and \({\bf I} \in \mathbb{R}^{d\times d}\) is the identity tensor. The differential operators \(\nabla\) (gradient) and \(\nabla\cdot\) (divergence) act with respect to the spatial variable \(\bsx\in\Omega\). 
To guarantee well-posedness of the problem \eqref{eq:EE}-\eqref{eq:sig}, we assume that 
\[0 < \mu_{\min} \le \mu \le \mu_{\max} < \infty~~{\rm and}~~
0 < \lambda_{\min} \le \lambda \le \lambda_{\max} < \infty
\quad \text{on }\Omega.\]
We are interested in the case where the material becomes nearly incompressible,  and so, the Poisson ratio $\nu$ of the elastic material approaches $1/2$, or equivalently, $\lambda_{\min}/\mu_{\max} \gg1$. 

In the special case, when $\lambda$ and $\mu$ are constant, in the nearly incompressible limit of the elasticity equation \eqref{eq:EE}, it is known that standard conforming finite element methods (FEMs) experience a degradation of convergence rates as the first Lam\'e parameter \(\lambda\) becomes large, a phenomenon referred to as volumetric locking. Scott and Vogelius \cite{ScottVogelius1985} demonstrated that on simplicial meshes, conforming polynomial spaces of degree \(p\ge4\) remain free of locking. In striking contrast, Babu\v{s}ka and Suri \cite{BabuskaSuri1992} proved that on quadrilateral meshes, conforming elements of any polynomial degree \(p\ge1\) are inevitably afflicted by locking. The comprehensive survey by Ainsworth and Parker \cite{AinsworthParker2022} analyzes the nuanced numerical pathologies, such as amplified error constants and spurious solution oscillations, that arise in the nearly incompressible regime.

For constant $\lambda$ and $\mu$, volumetric locking in standard conforming FEMs for the nearly incompressible elasticity problem \eqref{eq:EE} has catalyzed the formulation of numerous locking-free schemes. Nonconforming and mixed conforming-nonconforming FEMs (e.g., Arnold et al. \cite{ArnoldAwanouWinther2014}; Brenner and Sung \cite{BrennerSung1992}; Falk \cite{Falk1991}; Gopalakrishnan and Guzmán \cite{GopalakrishnanGuzman2011}; Lee, Lee, and Sheen \cite{LeeLeeSheen2003}; Mao and Chen \cite{MaoChen2008}), weak Galerkin approaches (e.g., Chen and Xie \cite{ChenXie2016}; Huo et al. \cite{HuoWangWangZhang2020}; Liu and Wang \cite{LiuWang2022}), discontinuous Galerkin formulations (e.g., Bramwell et al. \cite{BramwellDemkowiczGopalakrishnanQiu2012}; Di Pietro and Nicaise \cite{DiPietroNicaise2013}; Hansbo and Larson \cite{HansboLarson2002}; Cockburn et al. \cite{SoonCockburnStolarski2009}), and virtual element methods (e.g., Beirão da Veiga et al. \cite{VeigaBrezziMarini2013}; Edoardo et al. \cite{EdoardoStefanoCarloLuca2020}; Zhang et al. \cite{ZhangZhaoYangChen2019}) have successfully averted locking, albeit with greater algorithmic and analytical complexity. Extending many of these techniques to heterogeneous materials, where the Lam\'e coefficients vary with \(\bsx \), introduces additional challenges. In \cite{Mustaphaetal2024}, the authors address this by substituting the Lam\'e parameter \(\lambda\) in the stiffness assembly with an appropriate mesh-dependent surrogate
$\lambda_h$ (with $\lambda_h <\lambda$),  
where \(h\) denotes the maximum finite element mesh size and \(L\) is the diameter of \(\Omega\). This simple yet effective modification preserves the efficiency of the lowest-order conforming FEMs while substantially alleviating volumetric locking.

Deep artificial neural networks (DNNs) have garnered substantial attention as solvers for both forward and inverse partial differential equations (PDEs), owing to their capacity for high-dimensional approximation, the integration of automatic differentiation, and the maturity of open-source machine learning infrastructures. In particular, Physics-Informed Neural Networks (PINNs) incorporate the governing PDE residual directly into the loss functional, yielding a mesh-free paradigm that has experienced exponential growth in applications, surpassing 1,300 publications in 2021 alone \cite{Cuomoetal2022}. Comprehensive reviews have critically examined the theoretical foundations, algorithmic limitations, and emerging use cases of PINNs \cite{BlechschmidtErnst2021,CaiMaoWang2021,KarniadakisKevrekidisLu2021}, while advancing variants such as the variational hp-VPINN framework \cite{KharazmiZhangKarniadakis2021}, the variable-scaling PINN \cite{vs_pinn}, and physics-constrained neural formulations \cite{SunGaoPan2020a,ZhuZabarasKoutsourelakis2019} have demonstrated enhanced approximation fidelity and stability in challenging regimes.

PINNs were introduced in $2017$ by Raissi et al. as a new class of solvers in a two-part article \cite{RaissiPerdikarisKarniadakis2017c,RaissiPerdikarisKarniadakis2017d}. Later on, and for solving nonlinear PDEs,  the same authors created PINNs \cite{RaissiPerdikarisKarniadakis2019}, which can handle both forward problems of estimating the solutions of governing mathematical models and inverse problems. PINNs approximate solutions by training a neural network to minimize a loss function; it includes terms reflecting both initial and boundary conditions in addition to the PDE residual at selected points in the domain (so-called collocation points).  Incorporating a network for the PDE residual that encodes the governing physics equations is a significant novelty with PINNs. For a given input in the integration domain, PINNs produce an estimated solution of a differential equation after training. The basic concept behind PINN training is that it can be thought of as an unsupervised strategy that does not require labelled data, such as results from prior simulations or experiments. The PINN algorithm is essentially a mesh-free technique that finds PDE solutions by converting the problem of directly solving the governing equations into a loss function optimization problem. Owing to this,  PINNs can handle PDEs in domains with complicated geometries or in high dimensions. PINNs allow solutions to be made differentiable using analytical gradients and provide an easy way to solve forward and inverse problems using the same optimization problem (with minimal modifications). These can be considered some of the advantages of PINNs over conventional methods. 

Despite the growing success of PINNs in solving a wide range of PDEs, including two-dimensional elasticity equations with constant Lam\'e coefficients \cite{AlmeidaSilvaJr2023,ChenGu2023,EskinDavydovGurevaMalkhanovSmorkalov2024,GuoHaghighat2022,HaghighatRaissiMoure2021,RoyBoseSundararaghavanArroyave2023}, their application to nearly incompressible elasticity problems remains highly challenging. Similar to classical low-order conforming FEMs, PINNs suffer from locking or divergence instability. This instability manifests as a severe loss of accuracy and robustness, preventing the network from properly learning the displacement and stress fields. Our investigations reveal that the root cause of this issue lies in the intrinsic imbalance of the governing equations, where the volumetric ($\lambda\,\bigl(\nabla\cdot \vu\bigr)\,{ \bf I}$) and deviatoric ($2\mu\,\veps(\vu)$) terms of the stress tensor are scaled by vastly different magnitudes as the ratio of the Lam\'e coefficients $\lambda/\mu$ becomes very large.  For constant $\lambda$ and $\mu,$ this phenomenon is illustrated in Figure \ref{Locking_PINN}, where we train neural networks to solve the linear elasticity equations using PINNs with different values of $\lambda$ and with $\mu=1.$ A more detailed discussion of this phenomenon will be provided in Section \ref{sec_heuristic}. 
\begin{figure}
    \centering
    \begin{subfigure}{0.30\textwidth}
        \centering
        \includegraphics[width=\linewidth]{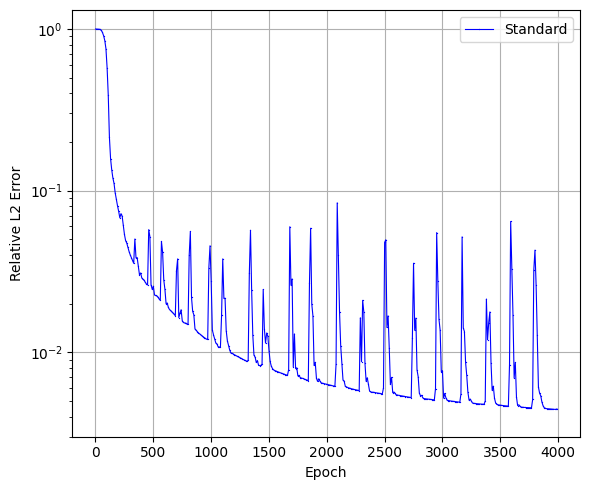}
        \caption{$\lambda=10$}
    \end{subfigure}   
    \begin{subfigure}{0.30\textwidth}
        \centering
        \includegraphics[width=\linewidth]{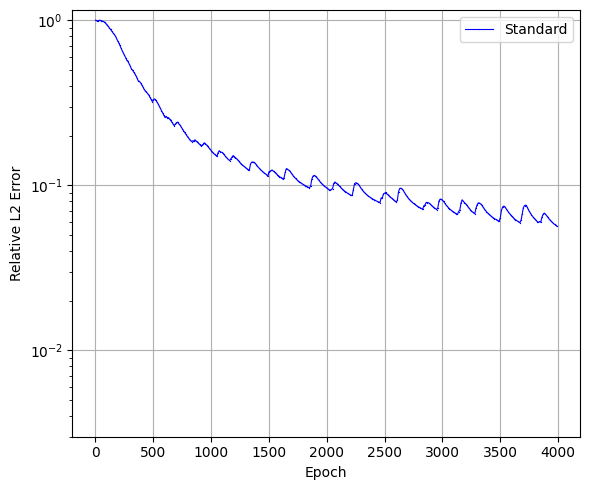}
        \caption{$\lambda=10^2$}
    \end{subfigure}    
    \begin{subfigure}{0.305\textwidth}
        \centering
        \includegraphics[width=\linewidth]{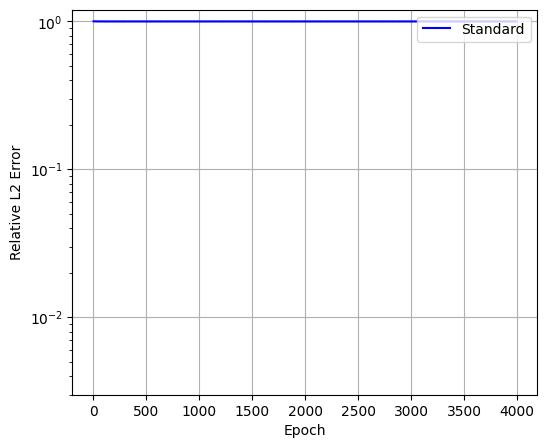} 
        \caption{$\lambda=10^5$}
    \end{subfigure}
    \caption{The locking phenomenon in solving linear elasticity equations using PINNs}
    \label{Locking_PINN}
\end{figure}

In this work, and to overcome the aforementioned challenge in PINNs, we introduce a novel decomposition-based PINN framework that reformulates the nearly incompressible elasticity equations into balanced coupled subsystems. This decomposition effectively mitigates the ill-conditioning that causes locking and allows the network to learn each component with comparable accuracy. The solution can be obtained via PINNs by minimizing the network’s loss function that comprises the residual error of the decomposed governing PDEs and the boundary conditions. The proposed approach provides a stable and efficient alternative to standard PINNs, capable of accurately capturing the behavior of nearly incompressible materials in two and three dimensions, including constant (homogeneous material), variable (inhomogeneous material), and random Lam\'e coefficients.


The remainder of this paper is organized as follows. Section \ref{sec_PINN} provides a brief overview of the PINN frameworks, including its parametric extension, which serves as the foundation of our approach. In Section \ref{sec_heuristic}, we present a heuristic explanation of why the standard PINN fails to solve linear elasticity when $\lambda$ is large. Section \ref{sec_method} then describes the proposed method to resolve this problem in detail, followed by Section \ref{sec_theory}, which establishes a theoretical convergence analysis. Section \ref{sec: numerical section} is devoted to demonstrating the effectiveness of the proposed modified PINNs through a series of numerical experiments on representative problems of the form \eqref{eq:EE}–\eqref{eq:sig} in both two and three dimensions. Moreover, we will conduct the experiment for the case of variable Lam\'e, as well as the case of uncertainties in the Lam\'e coefficients. Finally, some concluding remarks will follow in Section \ref{sec_conclusion}.
 
\section{Physics-informed neural networks}\label{sec_PINN}

\begin{figure}
    \centering
    \begin{subfigure}{0.8\textwidth}
        \centering
        \includegraphics[width=\linewidth]{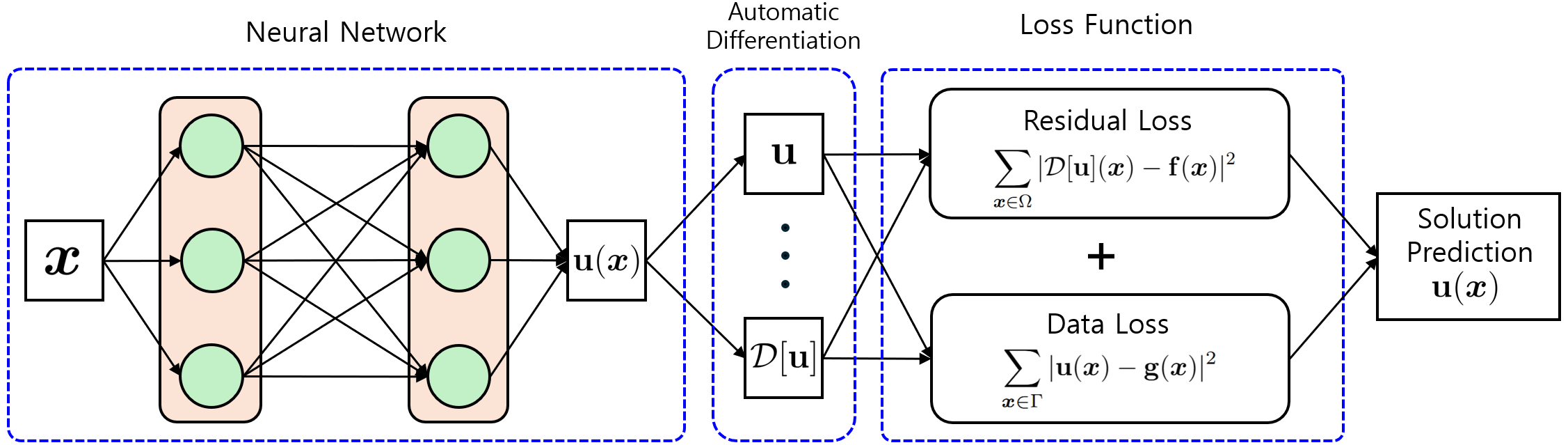}
        \caption{Standard PINN}
        \label{spinn}
    \end{subfigure} 
    \begin{subfigure}{0.8\textwidth}
        \centering
        \includegraphics[width=\linewidth]{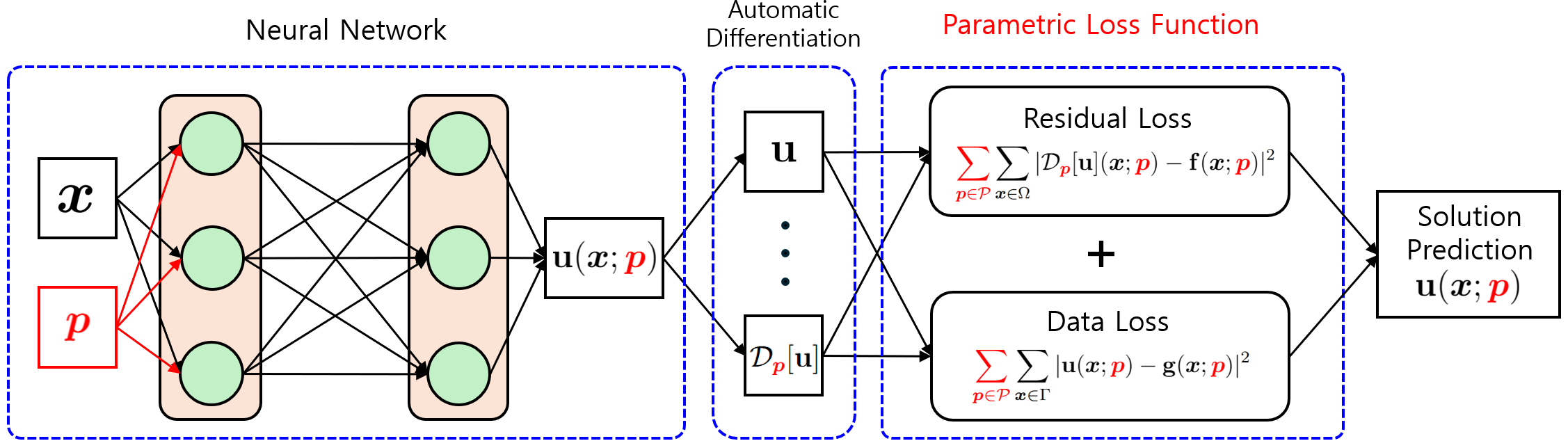} 
        \caption{Parametric PINN}
        \label{ppinn}
    \end{subfigure}
\caption{Comparison of schematic diagrams of PINN and parametric PINN.}
\label{comp_pinn}
\end{figure}

In this section, we provide a concise overview of the PINN framework, along with its extensions for parametric problems, which will be employed later in the paper. We begin by introducing the standard formulation of PINNs, which serves as the baseline approach \cite{RaissiPerdikarisKarniadakis2019}. Then, we introduce the parametric PINN, a novel approach that facilitates rapid inference of solutions under problems with varying parameters, such as variable PDE coefficients, boundary conditions, initial conditions, or external forcing, without necessitating retraining. In this work, the Young's modulus $E$ and Poisson's ratio $\nu$  will be used as input coefficients, which determine the given elasticity equations. A comparative illustration of these methodologies is presented in Figure \ref{comp_pinn}.

\subsection{Vanilla PINN}
First, we shall provide a brief description of PINNs \cite{RaissiPerdikarisKarniadakis2019}. Let us consider the following problem defined on a bounded open domain $\Omega\subset\R^d$ ($d\ge 1$) with a boundary $\Gamma=\partial\Omega$:
\begin{equation}\label{main_eq}
\begin{aligned}
\mathcal{D}[\vu](\bsx)&=\vf(\bsx), &&\boldsymbol{x}\in\Omega,\\
\vu(\boldsymbol{x})&=\vg(\boldsymbol{x}),&&\boldsymbol{x}\in\Gamma,
\end{aligned}
\end{equation}
where $\mathcal{D}$ is a given differential operator and $\vu:\overline{\Omega}\rightarrow\R^m$ (for some $m\in\mathbb{N}$) is the unknown solution, which is possibly multi-dimensional. Indeed, for time-dependent problems, the formulation of PINNs can be simplified by treating the temporal variable $t$ as an additional coordinate of $\bsx$, that is,  $\Omega$ may represent the spatio-temporal domain. The core concept of PINNs is to approximate the solution $\vu(\bsx)$ with a neural network 
$\vu(\bsx;\tth)$, where $\tth$ denotes the trainable parameters which will be described in more detail in Section \ref{sec_theory}, and to learn $\tth$ by minimizing a loss that enforces the governing PDE residual in the domain $\Omega$ together with the boundary conditions on $\partial\Omega$. The model is trained by minimizing a {\textit{physics-informed loss function}}, which includes the residual of the given PDE within the domain and the discrepancies between the neural network and the boundary conditions. More specifically, the neural network is trained to minimize the total loss
\begin{equation}\label{total_loss}
    \mathcal{L}(\tth)=\delta_r\mathcal{L}_r(\tth)+\delta_b\mathcal{L}_b(\tth),
\end{equation}
with each loss function defined as follows:
\begin{equation}\label{res_loss}
\mathcal{L}_r(\tth)=\frac{1}{N_r}\sum_{i=1}^{N_r}\left|\mathcal{D}[\vu](\boldsymbol{x}_r^i;\tth)-\vf(\boldsymbol{x}_r^i)\right|^2,\quad\mathcal{L}_b(\tth)=\frac{1}{N_b}\sum_{j=1}^{N_b}\left|\vu(\boldsymbol{x}_b^j;\tth)-\vg(\boldsymbol{x}_b^j)\right|^2.
\end{equation}
Here, $N_r$ and $N_b$ represent the batch sizes, while $\delta_r$ and $\delta_b$ are the weights associated with the residual and boundary data, respectively. Typically, $\boldsymbol{x}_r^i$ and $\boldsymbol{x}_b^j$ are randomly drawn from the uniform distributions
$\mathcal{U}(\Omega)$ and $\mathcal{U}(\Gamma)$, respectively. All derivatives required in the loss function can be efficiently and accurately computed using reverse-mode {\textit{automatic differentiation}} \cite{auto_diff_1, auto_diff_2}. This approach ensures that information about the boundary conditions propagates throughout the physical domain while adhering to the physical laws expressed as PDEs. This feature is the key reason we refer to it as a `physics-informed' neural network. If any sensor data is available for training, i.e., if the ground-truth data $(\bsx^i_s,\bsy^i_s)$ where $\bsy^i_s=\vu(\bsx_s^i)$ for $i=1,2,\cdots,N_s$ is given, an additional supervised loss term 
$\mathcal{L}_{\rm sup}(\tth)= \frac{1}{N_s}\sum^{N_s}_{i=1}|\vu(\bsx^i_s;\tth)-\bsy^i_s|^2$ can be incorporated into the total loss. This approach is known to enhance model performance. However, in this paper, we do not consider this type of additional loss term. A schematic representation of the PINN is shown in Figure \ref{comp_pinn}(a).

\subsection{Parametric PINN}\label{sec:para_PINN}
While PINNs have demonstrated their effectiveness across diverse engineering applications, their standard formulation is inherently tailored to individual problem settings. As a result, any modification to the parameters defining PDE problems, such as boundary or initial conditions, necessitates a complete retraining of the neural network. To mitigate this drawback, we introduce a parametric extension of the PINN framework (see, e.g., \cite{pPINN}), designed to address this issue. This enhanced approach incorporates an additional mechanism that enables rapid prediction of approximate solutions for varying parameters without the need for retraining, thus significantly broadening the practical applicability of PINNs.

Consider a parameter space $\mathcal{P} \subset \mathbb{R}^k$ for some $k \in \mathbb{N}$, and let $\boldsymbol{p} \in \mathcal{P}$ denote a vector of coefficients that governs the characteristics of the associated PDE problem. In particular, in the later numerical experiments, we will design a model that utilizes the coefficients, including Young's modulus $E$ and Poisson's ratio $\nu$, which determine the Lam\'e coefficients of the governing equations, as input features. The principal novelty behind the parametric PINN framework is the incorporation of problem-defining coefficients as additional inputs to the neural network. This extension allows the training process to be conducted over both the spatial domain $\Omega$ and the parameter space $\mathcal{P} \subset \mathbb{R}^k$, enabling the network to infer solutions that vary with respect to both spatial variables and coefficients. 

To elaborate further, consider a parameter vector $\boldsymbol{p} \in \mathcal{P}$ which defines the PDE data, such as initial conditions,  boundary conditions, and source terms. In standard PINNs, one seeks to minimize the physics-informed loss function
\[
\min_{\tth} \mathcal{L}(\tth; \boldsymbol{p}),
\]
for a fixed choice of $\boldsymbol{p}\in\mathcal{P}$. In contrast, the parametric PINN aims to learn a solution map over the product domain $\mathcal{P}\times\Omega$ by minimizing a population loss function integrated over all possible parameters:
\[
\mathcal{L}_{\mathrm{para}}(\boldsymbol{\theta}) := \int_{\mathcal{P}} \mathcal{L}(\boldsymbol{\theta}; \boldsymbol{p}) \, \mathrm{d}\eta(\boldsymbol{p}),
\]
for some proper measure $\eta$ over $\mathcal{P}$. In practice, this integral is approximated by an empirical loss function. Specifically, we draw the random points $\boldsymbol{x}^i_r$, $\boldsymbol{x}^j_b$, $\boldsymbol{p}^k_r$ and $\boldsymbol{p}^\ell_b$ i.i.d. from uniform distributions $\boldsymbol{x}^i_r \sim \mathcal{U}(\Omega)$, $\boldsymbol{x}^j_b \sim \mathcal{U}(\Gamma)$ and $\boldsymbol{p}^k_r,\boldsymbol{p}^\ell_b \sim \mathcal{U}(\mathcal{P})$ to define
\begin{equation} \label{parametric_loss}
\begin{aligned}
\mathcal{L}_{\mathrm{para},\,r}(\boldsymbol{\theta}) &= \frac{1}{N_{pr}} \frac{1}{N_r} \sum_{k=1}^{N_{pr}} \sum_{i=1}^{N_r} \left| \mathcal{D}_{\boldsymbol{p}_r^k}[\vu](\boldsymbol{x}_r^i; \boldsymbol{p}_r^k; \boldsymbol{\theta}) - \vf(\boldsymbol{x}_r^i; \boldsymbol{p}_r^k) \right|^2, \\
\mathcal{L}_{\mathrm{para},\,b}(\boldsymbol{\theta}) &= \frac{1}{N_{pb}} \frac{1}{N_b} \sum_{\ell=1}^{N_{pb}} \sum_{j=1}^{N_b} \left| \vu(\boldsymbol{x}_b^j; \boldsymbol{p}_b^\ell; \boldsymbol{\theta}) - \vg(\boldsymbol{x}_b^j; \boldsymbol{p}_b^\ell) \right|^2,
\end{aligned}
\end{equation}
where $N_{pr}$ and $N_{pb}$ are the numbers of random samples from the parametric domain $\mathcal{P}$ for the PDE and the boundary condition respectively. Here, $\mathcal{D}_{\boldsymbol{p}}$ denotes the differential operator associated with the parameter $\boldsymbol{p}$, and the functions $\vf(\cdot\,;\boldsymbol{p})$ and $\vg(\cdot\,;\boldsymbol{p})$ represent parameter-dependent source terms and boundary data, respectively. The total parametric loss is then constructed as a weighted sum of the above two components:
\begin{equation} \label{parametric_loss1}
\mathcal{L}_{\mathrm{para}}(\boldsymbol{\theta}) = \delta_r \mathcal{L}_{\mathrm{para},\,r}(\boldsymbol{\theta}) + \delta_b \mathcal{L}_{\mathrm{para},\,b}(\boldsymbol{\theta}),
\end{equation}
where $\delta_r, \delta_b > 0$ are appropriately chosen weights balancing the residual and boundary terms. Once training is completed, the neural network $\vu(\cdot,\cdot\,; \boldsymbol{\theta}): \Omega \times \mathcal{P} \to \mathbb{R}^m$ provides solution predictions for any given spatial location and parameter vector, enabling real-time evaluation across a wide range of parameters. A schematic overview of the parametric PINN architecture is illustrated in Figure \ref{comp_pinn}(b).

\section{Effects of incompressibility: Volumetric locking in PINN training}\label{sec_heuristic}

In this section, we introduce a heuristic perspective on why the training of a standard PINN fails when applied to nearly incompressible linear elasticity equations. To simplify the heuristics, we assume that the material is homogeneous (i.e., $\lambda$ and $\mu$ are constant) and $\vg = {\bf 0}$. We consider a straightforward implementation of PINNs for linear elasticity, where, in this case, the loss function is of the form
\begin{align}\label{loss_heur}
\mathcal{L}(\tth) = & \frac{\delta_r}{N_r} \sum_{i=1}^{N_r} \left| \mathcal{D}(\vu)(\bsx^i_r; \tth) + \vf(\bsx^i_r) \right|^2 +  \frac{\delta_b}{N_b} \sum_{j=1}^{N_b} \left| \vu(\bsx^j_b; \tth) \right|^2, 
\end{align}
where $\mathcal{D}(\vu) = \nabla \cdot (\lambda \nabla \cdot \vu {\bf I}) + \nabla \cdot (\mu \veps(\vu))$.

Note that if we replace $\vf$ in \eqref{eq:EE} by another function $\vf_2$ which satisfies $\vf_2(\bsx_r^i) = \vf(\bsx_r^i)$ we obtain the same loss function \eqref{loss_heur}. However, the exact solution $\vu$ of \eqref{eq:EE} with right-hand sides $\vf$ differs from the exact solution $\vu_2$ of \eqref{eq:EE} with right-hand side $\vf_2$ in general. Hence, there is no unique solution $\vu$ of \eqref{eq:EE} associated with the minimizer of the loss function \eqref{loss_heur}. Hence, the discrete loss function \eqref{loss_heur} does not have a unique minimizer in general (for instance, both $\vu$ and $\vu_2$ minimize the loss function). In the following, we illustrate that, when $\lambda$ is large, the set of minimizers of \eqref{loss_heur} contains a `solution' which is close to zero, and that standard PINN approximates this solution.

When training a neural network using a gradient descent method, the starting point of the gradient descent method determines which minimizer the optimization procedure will converge to. We discuss this in the context where $\lambda$ becomes large. The gradient flow equation, given by
\begin{equation*}
\frac{\mathrm{d}}{\mathrm{d} t} \tth(t) = - \nabla_{\tth} \mathcal{L}(\tth),
\end{equation*}
models the training behavior of the coefficients $\tth$ of a neural network when the step-size of the training goes to $0$ such that the parameters $\tth$ depend continuously on the training time $t$. Explicitly,
\begin{equation}\label{loss_gradient}
\begin{aligned}
\nabla_{\tth} \mathcal{L}(\tth) & = \frac{2 \delta_r}{N_r} \sum_{i=1}^{N_r}  \left( \mathcal{D}(\vu)(\bsx^i_r; \tth) + \vf(\bsx^i_r) \right) \mathcal{D}( \nabla_{\tth} \vu)(\bsx_r^i; \tth) + \frac{2 \delta_b}{N_b} \sum_{j=1}^{N_b}  \vu(\bsx^j_b; \tth)  \nabla_{\tth} \vu(\bsx_b^j; \tth) \\ & = \lambda^2 \va + \lambda \mu \vb + \mu^2 \vc + \vd,
\end{aligned} 
\end{equation}
where
\begin{align*}
\va & = \frac{2 \delta_r}{N_r} \sum_{i=1}^{N_r}  [ \nabla \cdot (\nabla \cdot ( \vu(\bsx_i^{e}; \tth)  {\bf I}) ] [\nabla \cdot (\nabla \cdot (\nabla_{\tth} \vu(\bsx_i^{e}; \tth) ) {\bf I}) ],  \\
\vb & = \frac{2 \delta_r}{N_r} \sum_{i=1}^{N_r}  [ \nabla \cdot (\nabla \cdot ( \vu(\bsx_i^{e}; \tth)  {\bf I}) ] [ \nabla \cdot \veps( \nabla_{\tth} \vu(\bsx_i^{e}; \tth) ) ] +  [\nabla \cdot (\nabla \cdot (\nabla_{\tth} \vu(\bsx_i^{e}; \tth) ) {\bf I}) ] [ \nabla \cdot \veps(\vu(\bsx_i^{e}; \tth)) ], \\
\vc & = \frac{2 \delta_r}{N_r} \sum_{i=1}^{N_r} [ \nabla \cdot \veps(\vu(\bsx_i^{e}; \tth)) ] [ \nabla \cdot \veps( \nabla_{\tth} \vu(\bsx_i^{e}; \tth)) ], \\
\vd & =    \frac{2 \delta_b}{N_b} \sum_{j=1}^{N_b} \vu(\bsx^j_b; \tth)  \nabla_{\tth} \vu(\bsx_b^j; \tth).
\end{align*}

Assume that the sampling points $\bsx_r^i, \bsx_b^j$ and the initial neural network for the PINN training have been chosen randomly, irrespective of the values of $\lambda$ and $\mu$.\footnote{Note that the solution $\vu^\ast$ of the linear elasticity equation changes if we change $\lambda$ and $\mu$.} This random choice induces distributions on the vectors $\va, \vb, \vc, \vd$, which are again independent of $\lambda$ and $\mu$. 
At the initialization of the gradient descent algorithm, the direction of the gradient of the loss function $\nabla_{\tth} \mathcal{L}$ depends, in particular, on $\lambda, \mu, \va, \vb, \vc, \vd$. As $\lambda$ increases, the direction $\va$ becomes more dominant. To illustrate this, consider
\begin{equation}\label{standard_grad}
\frac{1}{\lambda^2} \nabla_{\tth} \mathcal{L} = \va + \frac{\mu}{\lambda} \vb + \frac{\mu^2}{\lambda^2} \vc + \frac{1}{\lambda^2} \vd.
\end{equation}
As $\lambda$ increases, the direction of the gradient descent algorithm is increasingly aligned with the direction of $\va$ as the term $\frac{\mu}{\lambda} \vb + \frac{\mu^2}{\lambda^2} \vc$ becomes negligible (the importance of the boundary term $\frac{1}{\lambda^2} \vd$ can be adjusted via the weight $\delta_b$). This means that with increasing probability, in the optimization procedure, $\tth$ is modified  in a direction which minimizes the loss function
\begin{equation}\label{loss_reduced}
\overline{\mathcal{L}} = \frac{\delta_r}{N_r} \sum_{i=1}^{N_r} | \nabla \cdot (\nabla \cdot \vu(\bsx_r^i; \tth) {\bf I}) |^2 + \frac{\delta_b}{N_b} \sum_{j=1}^{N_b} |\vu(\bsx_b^j; \tth)|^2,
\end{equation}
since $\nabla_{\tth} \overline{\mathcal{L}} = \lambda^2 \va + \vd$. However, minimizing the loss function \eqref{loss_reduced} corresponds to approximating a function $\bar{\vu}$ which minimizes the loss function
\begin{equation}\label{eq_heur}
\delta_r \left \| \nabla \cdot (\nabla \cdot \vu {\bf I}) \right\|_{L_2(\Omega)}^2 + \delta_b \left\| \vu \right\|_{L_2(\partial \Omega}^2.
\end{equation}
However, it is easy to verify that the minimizer of \eqref{eq_heur} is the function which is $\vzero$ everywhere. In other words, the basin of attraction for gradient descent methods for the $\vzero$ solution increases with $\lambda$, so that for nearly incompressible materials the standard PINN leads to the $\vzero$ solution unless the initial guess is already very close to the true solution. In particular, we observe this kind of behavior in the numerical experiments for large $\lambda$ and zero Dirichlet boundary conditions using a standard PINN approach in Section~\ref{SectionEx1}.

\begin{figure}
    \centering    \includegraphics[width=0.4\textwidth]{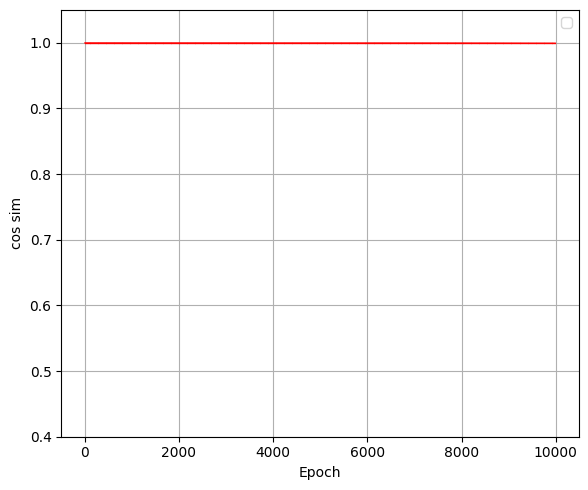}
    \caption{Cosine similarity between $\nabla_{\tth}\mathcal{L}$ and $\nabla_{\tth}\overline{\mathcal{L}}$ during training. This demonstrates that minimizing the loss $\mathcal{L}$ proceeds almost identically to minimizing the loss $\overline{\mathcal{L}}$, thereby explaining the occurrence of the locking phenomenon in PINN training.}
    \label{cos_sim}
\end{figure}

To make this heuristic more concrete, we conducted a numerical experiment. Specifically, we initialized two neural networks identically and trained them separately using loss functions $\mathcal{L}$ in \eqref{loss_heur} and $\overline{\mathcal{L}}$ in \eqref{loss_reduced} respectively, while observing their training dynamics. To examine the directions of the parameter updates during training, we recorded quantities $\nabla_{\tth}\mathcal{L}$ and $\nabla_{\tth}\overline{\mathcal{L}}$ at each epoch and measured the cosine similarity between them. Interestingly, as shown in Figure \ref{cos_sim}, the cosine similarity between $\nabla_{\tth}\mathcal{L}$ and $\nabla_{\tth}\overline{\mathcal{L}}$ remained close to 1 throughout the training process. This indicates that, when $\lambda$ is large, the training of a standard PINN proceeds almost as if it were minimizing the loss function $\overline{\mathcal{L}}$. However, since the solution obtained by minimizing $\overline{\mathcal{L}}$ differs from the true solution of the original PDE, this explains why the training of a standard PINN fails to converge properly in the nearly incompressible regime. In other words, when solving linear elasticity problems with large $\lambda$ using PINNs, an imbalance among the terms in the loss function leads to a phenomenon similar to locking, where the training does not proceed correctly. Much like in classical FEM, overcoming this locking effect remains a significant computational challenge. In the next section, we propose a method to address this issue, thereby enabling PINNs to effectively compute solutions in the nearly incompressible regime.

\section{Robust PINNs for nearly incompressible equations}\label{sec_method}
Using PINNs for solving the elasticity problem \eqref{eq:EE}-\eqref{eq:sig} for the case of nearly incompressible materials leads to a non-robust solution (locking), and the training of the PINN completely fails (see, e.g., Figure \ref{Locking_PINN}). As discussed earlier, this phenomenon is mainly due to the extreme imbalance between the coefficients $\mu$ and $\lambda$. To avoid the non-robustness issue of PINNs, we separate the effects of these coefficients by decomposing the given equations so that we can efficiently handle the large parameter $\lambda$. To be more specific, if we  set $\vu=\widehat{\vu}+\widetilde{\vu}$ for some $\widehat{\vu}$ and $\widetilde{\vu}$ in suitable function spaces, we may rewrite \eqref{eq:EE}-\eqref{eq:sig} as: with $\vsigma_\lambda(\vv) = \lambda\,\bigl(\nabla\cdot \vv\bigr)\,{ \bf I}$ and $\vsigma_\mu(\vv) =  2\mu\,\veps(\vv),$
\begin{equation}\label{eq:LL1}
\begin{aligned}
-\nabla \cdot \vsigma_\lambda(\vu) - \nabla \cdot \vsigma_\mu(\vu) =  -\nabla \cdot \vsigma_\lambda(\widehat{\vu}+\widetilde{\vu})-\nabla\cdot
\vsigma_\mu(\widehat{\vu}+\widetilde{\vu}) &= \vf \quad \,\text{in $\Omega$},\\
\vu = \widehat{\vu}+\widetilde{\vu}&=\vg\quad \text{on $\Gamma$}.
\end{aligned}
\end{equation}
This motivates us to split \eqref{eq:LL1} into two different systems of PDEs describing $\widehat{\vu}$ and $\widetilde{\vu}$ as: 
\begin{equation}\label{eq:dec_1}
\begin{aligned}
 -\nabla \cdot \vsigma_\lambda(\widehat{\vu}) &= \widehat{\vf}_\lambda \quad \,\,\,\, \text{in $\Omega$},\\
 -\nabla\cdot
\vsigma_\mu(\widehat{\vu})&= \widehat{\vf}_\mu \quad \,\,\,\,\text{in $\Omega$},\\
\widehat{\vu}&=\widehat \vg\quad\,\,\,\,\,\text{on $\Gamma$},
\end{aligned}
\end{equation}
and
\begin{equation}\label{eq:dec_2}
\begin{aligned}
 -\nabla \cdot \vsigma_\lambda(\widetilde{\vu})&= \boldsymbol{0} 
 \quad\,\,\,\, \,\,\text{in $\Omega$},\\
 -\nabla\cdot
\vsigma_\mu(\widetilde{\vu}) &=\widetilde{\vf}_\mu \quad \,\,\,\,\,\text{in $\Omega$},\\
\widetilde{\vu}&=\widetilde{\vg}\quad\,\,\,\,\,\,\text{on $\Gamma$},
\end{aligned}
\end{equation}
with
\begin{equation}\label{eq:dec_4}
\vf = \widehat{\vf}_\lambda +\widetilde{\vf}_\mu + \widehat{\vf}_\mu\quad{\rm and}\quad \vg =  \widehat \vg+\widetilde \vg.
\end{equation}
Let $(\widehat{\vu}, \widetilde{\vu})$ be any pair of solutions to \eqref{eq:dec_1} and \eqref{eq:dec_2}. Then it is easy to verify that $\widehat{\vu}+\widetilde{\vu}$ satisfies the original problem \eqref{eq:EE}. Since we assume that \eqref{eq:EE} has a unique solution $\vu$, the sum $\widehat{\vu}+\widetilde{\vu}$ $(=\vu)$ is uniquely determined, regardless of the decomposition $\widehat{\vu}$ and $\widetilde{\vu}$. As we can see from the decomposed system \eqref{eq:dec_1}-\eqref{eq:dec_2}, the original forcing term
$\vf$ is separated into three terms: $ \widehat{\vf}_\lambda$ is the term related to $\lambda$, and $\widehat{\vf}_\mu, \widetilde{\vf}_\mu$ are the terms related to $\mu$. By considering the terms associated with $\lambda$ and $\mu$ separately, we can expect that the imbalance between $\lambda$ and $\mu$, which causes locking, will be resolved, and the negative impact of the large magnitude of $\lambda$ for computing solutions is relatively alleviated. In fact, we can divide \eqref{eq:dec_2}$_1$ by $\lambda$ (or $\|\lambda\|_{L^\infty}$ when $\lambda$ is a variable function) and hence $\lambda$ does not influence this equation at all. Thus $\lambda$ does not directly appear in \eqref{eq:dec_2}.  Note here that the functions $\widehat{\vf}_\lambda, \widehat{\vf}_\mu, \widetilde{\vf}_\mu$ and $\widehat{\vg}, \widetilde{\vg}$ are not known a priori and are determined by solving inverse problems, as part of the training of the neural network. 

Regarding the existence of a solution of \eqref{eq:dec_1} and \eqref{eq:dec_2}, we first note that under certain assumptions on the boundary $\Gamma$ of the body $\Omega$ and for sufficiently regular $\mu, \widehat{\vf}_\mu, \widehat{\vg}$, it is known from the linear elliptic PDE theory (see, e.g., \cite{Trudinger}) that \eqref{eq:dec_1}$_2$-\eqref{eq:dec_1}$_3$ has a solution $\widehat{\vu}(\widehat{\vf}_\mu, \widehat{\vg})$. By plugging $\widehat{\vu}(\widehat{\vf}_\mu, \widehat{\vg})$ into \eqref{eq:dec_1}$_1$ we obtain a function $\widehat{\vf}_\lambda$. Now set $\widetilde{\vu} = \vu - \widehat{\vu}(\widehat{\vf}_\mu, \widehat{\vg})$, where $\vu$ is the solution of \eqref{eq:LL1}. Then the linearity of $\vsigma_\lambda, \vsigma_\mu$ and the fact that $\widehat{\vf}_\lambda + \widehat{\vf}_\mu + \widetilde{\vf}_\mu = \vf$ implies that $\widetilde{\vu}$ satisfies \eqref{eq:dec_2}$_1$ and \eqref{eq:dec_2}$_2$. Since $\vu$ is $\vg$ on $\Gamma$ and $\widehat{\vu}$ is $\widehat{\vg}$ on $\Gamma$, we also obtain that $\widetilde{\vu}$ is $\vg - \widehat{\vg} = \widetilde{\vg}$ on $\Gamma$. Thus $\widetilde{\vu}$ satisfies \eqref{eq:dec_2}. 

Note here that, in the decomposition \eqref{eq:dec_1}-\eqref{eq:dec_2}, the sum $\widehat{\vu}+\widetilde{\vu}$ recovers the solution $\vu$ of the original problem \eqref{eq:EE}, yet the accompanying allocation of body forces and boundary data need not be fixed. Different admissible choices of $(\widehat{\vf}_\lambda, \widehat{\vf}_\mu, \widetilde{\vf}_\mu;\, \widehat{\vg}, \widetilde{\vg})$ can support distinct constructions of $\widehat{\vu}$ and $\widetilde{\vu}$ while leaving their sum unchanged. For example, if we let
\[\widehat{\vu}_2(\bsx)  = \widehat{\vu}(\bsx) + Q \bsx + \bsd\quad{\rm and}\quad 
\widetilde{\vu}_2(\bsx)  = \widetilde{\vu}(\bsx) - Q \bsx - \bsd,\]
where $\bsd \in \mathbb{R}^d$ and $Q $ is a skew-symmetric matrix, then $(\widehat{\vu}_2, \widetilde{\vu}_2)$ is also a solution of \eqref{eq:dec_1} and \eqref{eq:dec_2}. More generally, choosing $\widetilde{\vu}$ to satisfy \eqref{eq:dec_2}$_1$ and \eqref{eq:dec_2}$_3$ determines $\widetilde{\vf}_\mu$ via \eqref{eq:dec_2}$_2$, after which setting $\widehat{\vu} := \vu-\widetilde{\vu}$ enforces (4.2). Thus, multiple load-boundary configurations can realize the same overall solution through the decomposition.

\begin{figure}
    \centering    \includegraphics[width=1\textwidth]{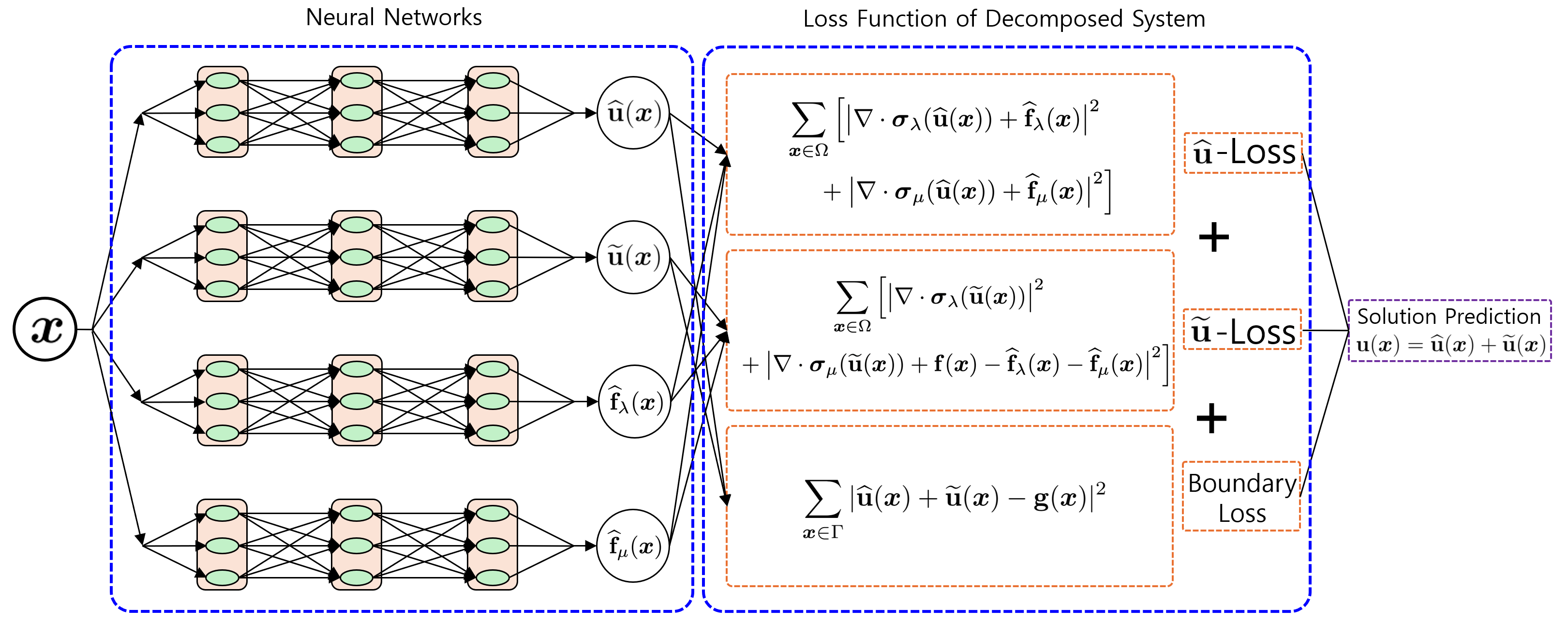}
    \caption{A schematic diagram of the proposed method using a decomposition. Three independent neural networks are utilized for the three intermediate functions $\widehat{\vu}$, $\widetilde{\vu}$, and $\widehat{\vf}$, and the loss functions are defined based on the decomposition \eqref{eq:dec_1}-\eqref{eq:dec_2}.}
    \label{PINN_scheme_decomp}
\end{figure}

We shall now employ PINNs to solve the decomposed systems \eqref{eq:dec_1}-\eqref{eq:dec_2}. Owing to the flexibility of PINNs in incorporating problem-specific physics through the design of appropriate loss functions, this framework enables a direct mathematical formulation of the previously discussed scenario. The auxiliary functions $\widehat{\vf}_\lambda, \widehat{\vf}_\mu, \widetilde{\vf}_\mu$ as well as the boundary functions $\widehat{\vg}, \widetilde{\vg}$ will themselves be modeled by neural networks, which will be automatically adjusted during the PINN training process. For the input value $\bsx$, we shall consider four independent fully-connected neural networks, where each neural network generates $\widehat{\vu}$, $\widetilde{\vu}$, and $\widehat{\vf}_\lambda, \widehat{\vf}_\mu$ as outputs, respectively. At first, we initialize the parameters of the network for $\widehat{\vf}_\lambda$, $\widehat{\vf}_\mu$ so that $\widehat{\vf}_\lambda=\frac{1}{3}\vf$ and $\widehat{\vf}_\mu=\frac{1}{3}\vf$ are satisfied. Note that our method’s performance is not sensitive to the choice $\frac{1}{3}$ here, and other values are also admissible. We then apply automatic differentiation to these output values in accordance with the systems in  \eqref{eq:dec_1} and \eqref{eq:dec_2}. Specifically, for random samples from the interior of the domain $\Omega$, we will compute the residuals of the equations in \eqref{eq:dec_1} and \eqref{eq:dec_2}, and for random samples at the boundary, we will set up the loss function so that the sum of $\widehat{\vu}$ and $\widetilde{\vu}$ equals $\vg$ so that the intermediate boundary value $\widehat\vg$ is automatically adjusted during the training. 

The above discussions lead us to consider the following form of the loss function
\begin{equation}\label{main_loss}
\begin{aligned}
\mathcal{L}\Big(\widehat{\vu},\widetilde{\vu},\widehat{\vf}\Big)
=&\frac{1}{N_r}\sum^{N_r}_{i=1}\left( \delta_{r,1} \big|\nabla \cdot \vsigma_\lambda(\widehat{\vu})(\bsx^i_r) +  \widehat{\vf}_\lambda(\bsx^i_r)  \big|^2+  \delta_{r,2} \big|\nabla\cdot  \vsigma_\mu(\widehat{\vu})(\bsx^i_r) + \widehat{\vf}_\mu(\bsx^i_r) \big|^2 \right)\\
&+\frac{1}{N_s}\sum^{N_s}_{j=1}\left( \delta_{s,1} \big| \nabla \cdot  \vsigma_\lambda(\widetilde{\vu})(\bsx^j_s)  \big|^2 +  \delta_{s,2} \big|\nabla\cdot\vsigma_\mu(\widetilde{\vu})(\bsx^j_s) +\vf(\bsx^j_s) - \widehat{\vf}_\lambda(\bsx^j_s) - \widehat{\vf}_\mu(\bsx^j_s) \big|^2\right)\\
&+\frac{\delta_b}{N_b}\sum^{N_b}_{k=1}\big|\widehat{\vu}(\bsx^k_b)+\widetilde{\vu}(\bsx^k_b)-\vg(\bsx^k_b) \big|^2,
\end{aligned}
\end{equation}
where for some weight parameters $\delta_r, \delta_s, \delta_b$, we set $\delta_{r,1} = \delta_{r,2}  = \delta_{r}$ and $\lambda^2 \delta_{s,1} = \delta_{s,2} = \delta_s$. 
In the case of inhomogeneous $\lambda$, where $\lambda(\bsx) = \Lambda \tilde{\lambda}(\bsx)$ for some constant $\Lambda$, we use $\delta_{s,1} = \delta_s / \Lambda^2$. 

The interior points $\bsx^i_r$ and $\bsx^j_s$ are i.i.d. randomly sampled from the uniform distribution $\mathcal{U}(\Omega)$ and the boundary samples $\bsx^k_b$ are i.i.d. randomly chosen from $\mathcal{U}(\Gamma)$. The graphical illustration of the proposed method can be found in Figure~\ref{PINN_scheme_decomp}. In the following section, we demonstrate through various numerical experiments that the proposed method is effective for the training of PINN for this problem and successfully resolves the locking phenomenon.

\section{Error analysis}\label{sec_theory}

In the following, we analyze the approximation error under the assumption that the optimization procedure yields a neural network which approximates the true solution. The inner product on the space ${\bf L}^2(\Omega)=[L^2(\Omega)]^d$ is denoted by $\ip{\cdot,\cdot}$, and is the associated norm is written as $\|\cdot\|$, 
where for $\vv, \vw \in {\bf L}^2(\Omega)$,
\[\ip{\vv,\vw}=\int_{\Omega} \vv\cdot \vw\,{\rm{d}}\bsx ~~{\rm and}~~\|\vv\| = \sqrt{\ip{\vv,\vv}}.
\]
For the inner product and the norm on ${\bf L}^2(\partial \Omega)$, we use the notations $\ip{\vv,\vw}_{\partial\Omega}$ and $\|\cdot\|_{\partial \Omega}$.

Let ${\bf H}^2(\Omega)=[H^2(\Omega)]^d$, where $H^2(\Omega)$ is the standard Sobolev space. The norm on ${\bf H}^2(\Omega)$ is denoted by $\|\cdot\|_2.$ The error is measured using the fractional norm $\|\cdot\|_{1/2}$, defined by: $\|\vw\|^2_{1/2}=\|\vw\|^2+|\vw|_{1/2}^2$, where $|\cdot|_{1/2}$ is the (Slobodeckij) seminorm on the fractional Sobolev space \([H^{1/2}(\Omega)]^d\)  defined  as: 
\[
|\vw|^2_{1/2} = \int_\Omega \int_\Omega \frac{|\vw(\bsx) - \vw(\bsy)|^2}{|\bsx - \bsy|^{d + 1}} \, {\rm{d}}\bsx \, {\rm{d}}\bsy.
\]
Define the following energy functional 
$\calE: {\bf H}^2(\Omega) \to \R$ by
\begin{equation}\label{eq:Eu}
\calE (\vv) = \delta_r \| \nabla \cdot \vsigma(\vv)+ \vf \|^2 + \delta_b \|\vv-\vg\|^2_{\partial\Omega},
\end{equation}
where $\delta_r, \delta_b$ are constants as defined in
\eqref{total_loss}. Solving the linear elasticity equation \eqref{eq:EE} is equivalent to minimizing the energy functional $\calE$ over the Sobolev space ${\bf H}^2(\Omega)$. 

A neural network with $L$ layers, each layer has $n_\ell$ neurons, for $\ell=0,\ldots,L$ can be modeled as the following structure
\[
\tth = (W_1, \vb_1, W_2, \vb_2,\ldots, W_L,\vb_L),
\]
where $W_{\ell} \in \R^{n_{\ell} \times n_{\ell-1}}$, for
$\ell=1,2,\ldots,L$ are the weight matrices
and $\vb_\ell \in \R^{n_\ell}$ are the bias vectors.   
Given an input vector $\vz_0 =\bsx \in \R^{n_0}$, the output vector $\va_L$ of the network $\tth$ is defined by
\begin{equation}\label{eq:ThetaN}  
    \va_L = W_{L} \vz_{L-1} + \vb_L,
\end{equation}
where  $\vz_{\ell} = \rho(W_{\ell} \vz_{\ell-1} + \vb_\ell)$ for  $\ell=1,2,\ldots,L-1,$ and $\rho$ is some given activation function. The
neural network function with parameters $\tth$ and the activation function $\rho$ is the function
\begin{equation}
  \vv_{\tth}:\R^{n_0} \to \R^{n_L},
  \quad \bsx \mapsto \va_L(\bsx).
\end{equation}
We denote by $\Theta$ the parameter space which contains all possible values of $\tth$. It can be seen that $\Theta = \R^p$ for some integer $p$ which depends on the values $n_{\ell}$, for $\ell=0,1,\ldots,L$. The collection of all network functions with a certain parameter space $\Theta$ is denoted by $\calF_{\Theta}$.
The PINN approach to solve the linear elasticity equation \eqref{eq:EE} is to minimize the energy $\calE$ given in \eqref{eq:Eu} over the function class $\calF_{\Theta}$. More precisely, the network will minimize a loss functional $\calL: \Theta \to \R$ (see \eqref{total_loss}), such that $\calL(\tth) \approx \calE(\vu_{\tth}).$ Let $\tth^*$ minimize $\calL(\tth)$ inexactly and let $\vu_{\tth^*}$ be the corresponding neural network solution, then
\begin{equation}\label{eq:estE}
    \|\vu_{\tth^*} - \vu\|^2_{1/2}
    \lesssim  \calL(\tth^*) + \eta(\tth^*),
\end{equation}
where $\eta(\tth^*) = \calE(\vu_{\tth^*}) - \calL(\tth^*)$ is a quadrature error. Furthermore,
\begin{equation}\label{eq:Cea lemma}
    \|\vu_{\tth^*} - \vu\|^2_{1/2}
    \lesssim \kappa(\vu_{\tth^*})
    + \inf_{\vu_{\tth}\in \calF_{\Theta}}
    \|\vu_{\tth} - \vu\|^2_{2},
\end{equation}
where the optimization error $\kappa(\vu_{\tth^*}) =
\calE(\vu_{\tth^*}) - \inf_{\vu_{\tth}} \calE(\vu_{\tth})$. See \cite{ZeinhoferMasriMardal2024} for the proof of these error estimate results.
We aim next to show similar results for the proposed decomposed  PINN method. For convenience, we introduce the following notations: let 
\[\vsigma_D=\begin{bmatrix}
    \vsigma_\lambda\\\vsigma_\mu\end{bmatrix},\quad \widehat \vf_D=\begin{bmatrix}\,
    \widehat{\vf}_\lambda\, \\ \widehat{\vf}_\mu\end{bmatrix},\quad \widetilde  \vf_D=\begin{bmatrix}\,
    \,{\bf 0}\,\\ \widetilde{\vf}_\mu\, \end{bmatrix},\quad \vv_D=\begin{bmatrix}\,
    \widehat \vv\\ \widetilde \vv\, \end{bmatrix},\quad \vv_S=\widehat \vv+\widetilde \vv.\]
The loss function $\calL_D$ is defined as  in \eqref{main_loss} with $\delta_{r,1}=\delta_{r,2}=\delta_r$ and $\delta_{s,1}=\delta_{s,2}=\delta_s$. The associated energy functional
\begin{equation}\label{eq:Euvf}
\begin{aligned}
\calE_D(\vv_D) &= 
\delta_r\| \nabla \cdot \vsigma_D(\widehat \vv) + \widehat{\vf}_D\|^2 + 
\delta_s \|\nabla \cdot \vsigma_D(\widetilde \vv) + \widetilde {\vf}_D\|^2 
 +\delta_b \|\vv_S-\vg\|^2_{\partial\Omega},
\end{aligned}
\end{equation}
where it is understood that in the first two terms on the RHS of
equation~\eqref{eq:Euvf} the vectors are in $[L^2(\Omega)]^{2d}$.
Hence, our PINN approach is to solve two systems 
\eqref{eq:dec_1} and \eqref{eq:dec_2} simultaneously by minimizing $\calE_D$ over the function class $\calF_{\Theta} \times \calF_{\Theta}$. More precisely, the networks will minimize a loss functional
\[
\calL_D: \Theta \times \Theta \to \R, \quad \calL_D(\tth_D)=\calL(\tth_D) \approx 
\calE_D(\vu_{D,\tth}) = \calE_D(\widehat \vu_{\theta_1}, \widetilde \vu_{\theta_2}).
\]
Define the following bilinear operator  $a_D: [H^2(\Omega)]^{2d} \times [H^2(\Omega)]^{2d} \to \R$ as
\begin{equation}\label{eq:bilinear}
a_D(\vv_D,\vw_D) =   
\delta_r \ip{\nabla \cdot \vsigma_D(\widehat \vv),\nabla \cdot \vsigma_D(\widehat \vw) } 
+ \delta_s \ip{ \nabla \cdot\vsigma_D(\widetilde \vv),\nabla \cdot\vsigma_D(\widetilde \vw)}  
+\delta_b \ip{\vv_S, \vw_S}_{\partial\Omega}.
\end{equation}
Define the bilinear operator  $a: {\bf H}^2(\Omega) \times {\bf H}^2(\Omega) \to \R$ as 
\begin{equation}\label{bilinear a}
    a(\vv, \vw) = 
    \delta_r \ip{\nabla \cdot \vsigma(\vv),
    \nabla \cdot \vsigma(\vw)}
    + \delta_b \ip{\vv,\vw}_{\partial\Omega},
\end{equation}
Hence, by linearity and with the help of the Cauchy--Schwarz inequality,   
\begin{align*}
    a(\vv_S, \vv_S) &= 
    \delta_r \|\nabla \cdot (\vsigma_\lambda+\vsigma_\mu)(\widehat \vv+\widetilde \vv)\|^2
    + \delta_b \|\vv_S\|^2_{\partial\Omega}\\
     &\lesssim   
 \|\nabla \cdot \vsigma_\lambda(\widehat \vv)\|^2+\|\nabla \cdot \vsigma_\mu(\widehat \vv)\|^2 
+\|\nabla \cdot\vsigma_\lambda(\widetilde \vv)\|^2+\|\nabla \cdot\vsigma_\mu(\widetilde \vv)\|^2 
+\delta_b \|\vv_S\|^2_{\partial\Omega}\\
&\lesssim a_D(\vv_D,\vv_D).
\end{align*}
Therefore, by using the coercivity property of $a(\cdot, \cdot)$ in \cite[Lemma 3]{ZeinhoferMasriMardal2024},  we deduce that for $\widehat \vv,\,\widetilde \vv \in {\bf H}^2(\Omega)$, there holds
\begin{equation}\label{lem:coercive}
\|\vv_S\|^2_{1/2}=\|\widehat \vv+\widetilde \vv\|^2_{1/2}\lesssim a(\vv_S,\vv_S)
\lesssim a_D(\vv_D,\vv_D).    
\end{equation}
\begin{theorem}
Let $\tth^*_D$ minimize $\calL_D(\tth_D)$ inexactly, and let $\vu_{D,\tth^*}$  be the corresponding neural network solutions, then
\begin{equation}\label{eq:estEd}
    \|\vu_{S,\tth^*}  - \vu\|^2_{1/2}
    \lesssim  \calL_D(\tth^*_D) + \eta_D(\tth^*_D),
\end{equation}
where $\vu_{S,\tth^*}=\widehat{\vu}_{\theta_1^*} + \widetilde{\vu}_{\theta_2^*}$, $\vu =\vu_S= \widehat{\vu} + \widetilde{\vu}$ and $\eta_D(\tth^*_D) = 
\calE_D (\vu_{D,\tth^*}) -
\calL_D(\tth^*_D)$ 
is a quadrature error. Furthermore,
\begin{equation}\label{cea lem}
  \|\vu_{S,\tth^*} - \vu\|^2_{1/2}
    \lesssim
    \kappa(\vu_{D,\tth^*}) + 
    \inf_{ \vu_{D,\tth} \in \calF_{\Theta} \times \calF_{\Theta}} \|\vu_{D,\tth}-\vu_D\|^2_{2},
\end{equation}
with $\kappa(\vu_{D,\tth^*}) := \calE_D(\vu_{D,\tth^*})   
- \inf_{ \vu_{D,\tth} \in \calF_{\Theta} \times \calF_{\Theta}} \calE_D(\vu_{D,\tth})$.
\end{theorem}
\begin{proof}
Let $\widehat{\vu}$ and $\widetilde{\vu}$ solve the systems
\eqref{eq:dec_1} and \eqref{eq:dec_2}. 
Then $\calE_D(\vu_D) = 0$ and it follows from the definition
of the bilinear form \eqref{eq:bilinear} that
\begin{equation}\label{eq:energy1}
 \calE_D(\vu_{D,\tth^*})
- \calE_D(\vu_D)
=  a_D (  \vu_{D,\tth^*}-\vu_D,
\vu_{D,\tth^*}-\vu_D).
\end{equation}
Using the coercivity of the bilinear form $a_D(\cdot,\cdot)$, see 
\eqref {lem:coercive}, we obtain
\begin{equation}\label{eq:coercive}
\|\vu_{S,\tth^*} - \vu\|^2_{1/2}
\lesssim \calE_D (\vu_{D,\tth^*})
 = \calL_D(\tth_D^*) + \eta_D(\tth^*_D),
\end{equation}
 proving the estimate~\eqref{eq:estEd}. To prove the second estimate~\eqref{cea lem}, by using the definition of the bilinear form, we have, for any
$\vu_{D,\tth} \in [H^2(\Omega)]^{2d}$,
\[\calE_D(\vu_{D,\tth})-\calE_D(\vu_D)
= a_D ( \vu_{D,\tth^*}-\vu_D,\vu_{D,\tth^*}-\vu_D) 
\lesssim \|\vu_{D,\tth} -\vu_D\|^2_{2}.\]
Now using the above estimate
\begin{equation}\label{eq:now}
\begin{aligned}
\calE_D(\vu_{D,\tth})-\calE_D(\vu_D)
&= \calE_D(\vu_{D,\tth^*})   
- \inf_{\theta_1,\theta_2} \calE_D(\vu_{D,\tth})
+ \inf_{\theta_1,\theta_2} [\calE_D(\vu_{D,\tth})
-\calE_D(\vu_D)] \\
& \le
\kappa(\vu_{D,\tth^*}) + C \|\vu_{D,\tth} - \vu_D\|^2_{2}.
\end{aligned}
\end{equation}
The second estimate now follows from \eqref{eq:energy1}, \eqref{eq:coercive} and \eqref{eq:now}.
\end{proof}




\section{Numerical Experiments}\label{sec: numerical section}

\begin{figure}
  \centering
  \begin{minipage}{\textwidth}
    \centering
    \begin{subfigure}{0.25\textwidth}
      \includegraphics[width=\linewidth]{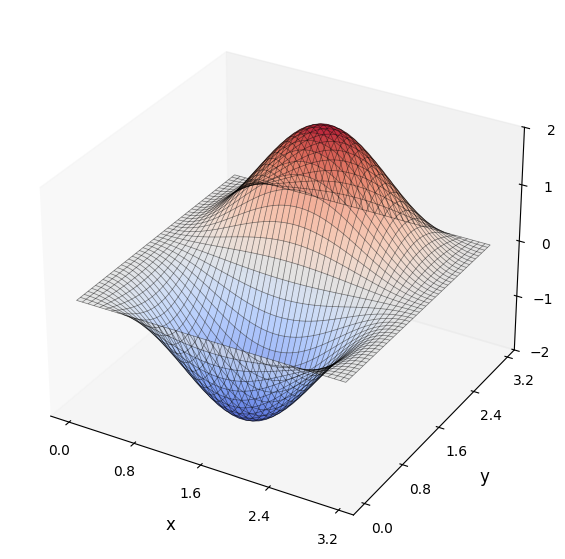}
      \caption*{$u_1$}
    \end{subfigure}
    \hspace{3mm}
    \begin{subfigure}{0.25\textwidth}
      \includegraphics[width=\linewidth]{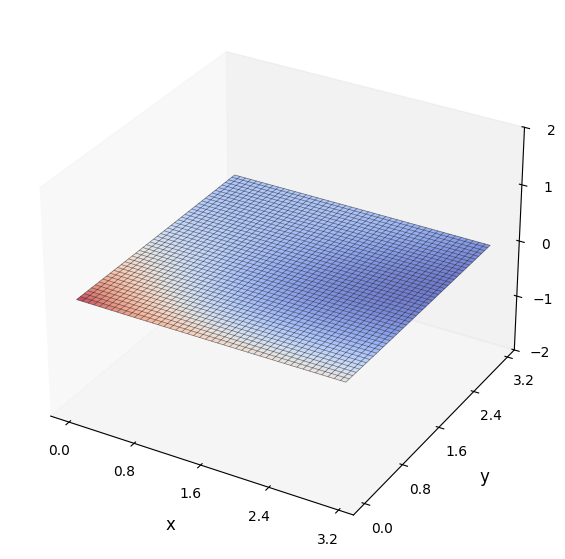}
      \caption*{$u_1$}
    \end{subfigure}
    \hspace{3mm}
    \begin{subfigure}{0.25\textwidth}
      \includegraphics[width=\linewidth]{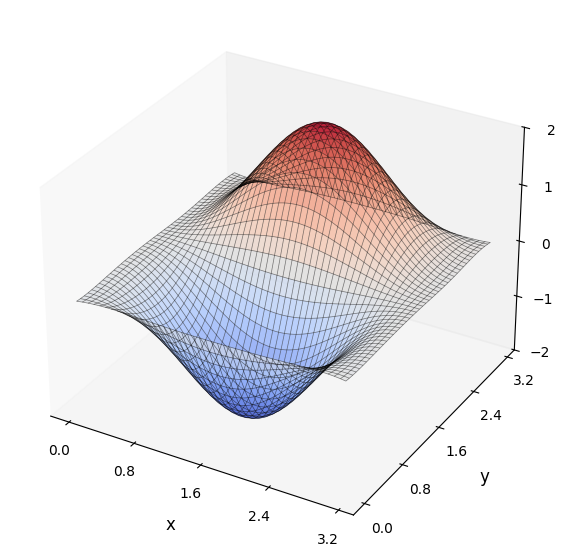}
      \caption*{$u_1$}
    \end{subfigure}
  \end{minipage}

  \begin{minipage}{\textwidth}
    \centering
    \begin{subfigure}{0.25\textwidth}
      \includegraphics[width=\linewidth]{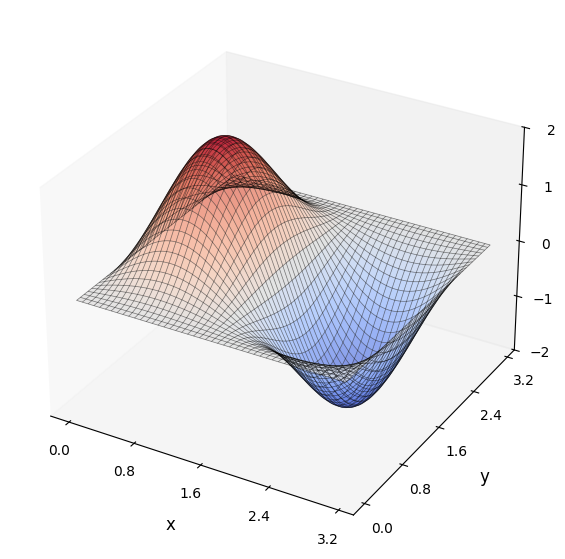}
      \caption*{$u_2$ \\ \vspace{2mm} (a) True solution}
    \end{subfigure}
    \hspace{3mm}
    \begin{subfigure}{0.25\textwidth}
      \includegraphics[width=\linewidth]{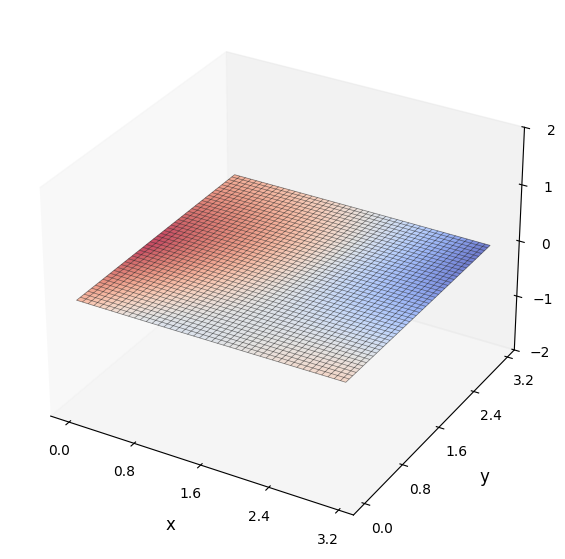}
      \caption*{$u_2$ \\  \vspace{2mm} (b) Standard PINN}
    \end{subfigure}
    \hspace{3mm}
    \begin{subfigure}{0.25\textwidth}
      \includegraphics[width=\linewidth]{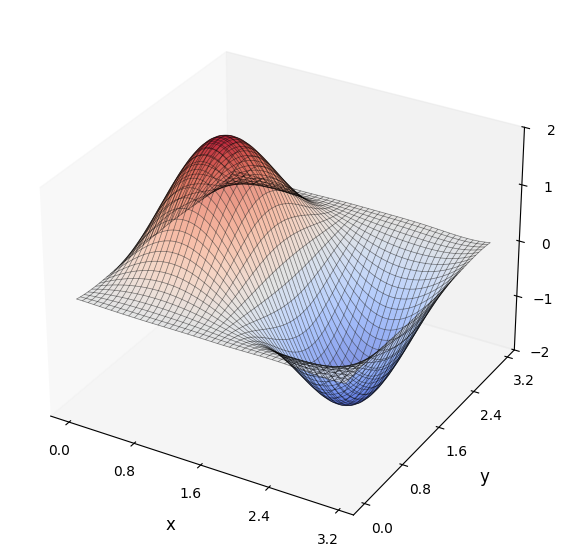}
      \caption*{$u_2$ \\  \vspace{2mm} (c) Our method}
    \end{subfigure}
  \end{minipage}
  \caption{A comparison of solutions in two dimensions when $\lambda=10^5$: the exact solution (a), the solution computed by the standard PINN (b), and the solution obtained by our method (c). The experiment shows that standard PINNs fail to approximate the true solution in this example.}
  \label{ex1_sol_pred}
\end{figure}

\begin{figure}
    \centering
    \begin{subfigure}{0.30\textwidth}
        \centering
        \includegraphics[width=\linewidth]{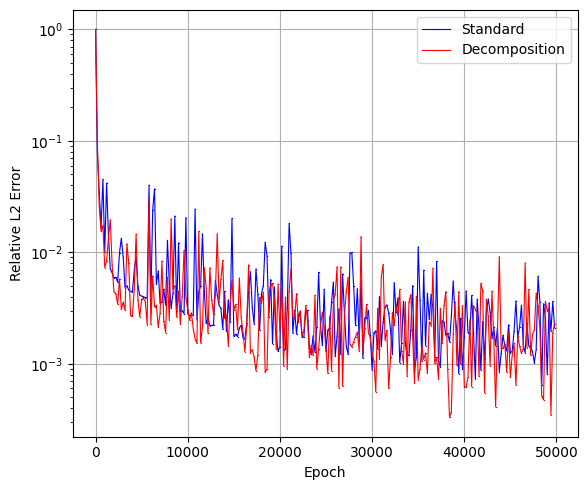}
        \caption{$\lambda=10$}
    \end{subfigure}   
    \begin{subfigure}{0.30\textwidth}
        \centering
        \includegraphics[width=\linewidth]{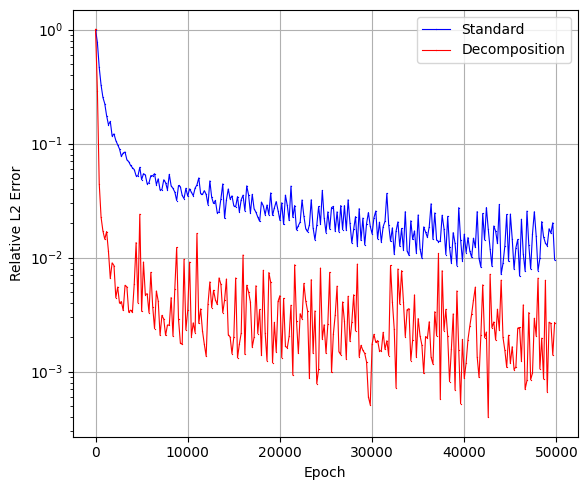}
        \caption{$\lambda=10^2$}
    \end{subfigure}    
    \begin{subfigure}{0.30\textwidth}
        \centering
        \includegraphics[width=\linewidth]{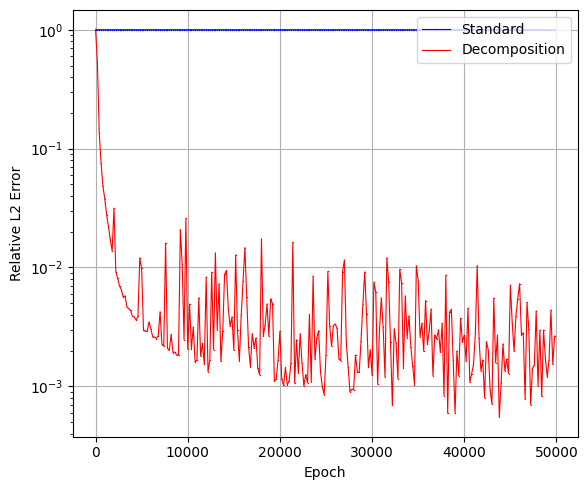} 
        \caption{$\lambda=10^5$}
    \end{subfigure}
    \caption{A comparison of learning curves with respect to the relative $L^2$-error between the standard PINN and our method for various values of $\lambda$ in two dimensions. The experiment shows that for $\lambda=10^5$, the standard PINN approach does not yield an approximation to the true solution in this example.}
    \label{ex1_learn_curve}
\end{figure}

In this section, we conduct a series of numerical experiments to assess the effectiveness of the proposed method. We compare a model trained with the standard PINN against a model trained using the proposed decomposition-based approach. To provide a comprehensive evaluation of the proposed method when the Lam\'e coefficients $\mu$ and $\lambda$ in the elasticity equations are constant,  including scenarios with large  $\lambda$ (nearly incompressible material), we perform some numerical experiments in both two and three dimensions. We then investigate the case of variable $\lambda$ and $\mu$ (inhomogeneous material). Finally, we will also conduct experiments on a parametric PINN, which enables the rapid prediction of solutions for varying Lam\'e coefficients without the need for retraining. To ensure a fair comparison, all other settings, including the model architecture, initialization, and optimization algorithm, remain unchanged.

\subsection{Example 1  (Two-dimensional case)}\label{SectionEx1}
We set $\Omega=(0,\pi)^2$ and $\mu=1,$  and choose the body force $\vf$ so that the true solution of \eqref{eq:EE} is
\begin{equation}\label{ex1}
\vu(\boldsymbol{x})=\begin{bmatrix}  u_1(x,y)\\  u_2(x,y)\end{bmatrix}
=\begin{bmatrix}
\bigl(\cos(2x)-1\bigr)\sin(2y)\\ \bigl(1-\cos(2y)\bigr)\sin(2x)
\end{bmatrix}
+\frac{\sin(x)\sin(y)}{\lambda}\begin{bmatrix}1\\ 1\end{bmatrix},    
\end{equation}
where $\bsx=(x,y) \in \Omega$. For a baseline model, we implement a neural network consisting of four hidden layers, each containing 50 neurons, activated by the GELU function. For training, as described before, we adopt the decomposed loss function \eqref{main_loss} with weights $\delta_r= 0.05$, $\delta_s= 1$ and $\delta_b= 20$, and number of sampling points $N_r=5000$, $N_s=5000$ and $N_b=400$. We use the Adam optimizer as an optimization algorithm.

Figure \ref{ex1_sol_pred} illustrates the solutions predicted by a standard PINN and our proposed method when $\lambda=10^5$ (nearly incompressible material). As shown in the figure, the solution obtained using the original PINN is close to zero with $1.01$e$0$ relative $L^2$-error, failing to accurately predict the solution of the given problem. In contrast, the solution computed using our proposed method closely approximates the exact solution with $1.65$e$-3$ relative $L^2$-error. To compare the learning curves for different values of $\lambda$, we conducted training using both the standard PINN and our proposed method for various values of $\lambda$, computing the relative $L^2$-error at each training epoch. The results of this experiment are presented in Figure \ref{ex1_learn_curve}. As observed from Figure \ref{ex1_learn_curve} (a), when $\lambda=10$ (no locking is anticipated), there is no significant difference in training efficiency between the original PINN and our proposed method. However, as $\lambda$ is getting large,  it is clear from  (b) and (c) of Figures \ref{ex1_learn_curve} that our method demonstrates superior training efficiency and accuracy compared to the standard PINN. It is noteworthy that, as we can see from Figure \ref{ex1_learn_curve} (c) (when $\lambda=10^5$, the material becomes nearly incompressible), the standard PINN fails to train due to the non-robustness phenomenon in the nearly incompressible case, whereas our proposed method continues to achieve effective training with high accuracy. 

\begin{figure}
  \centering
  \begin{minipage}{\textwidth}
    \centering
    \begin{subfigure}{0.25\textwidth}
      \includegraphics[width=\linewidth]{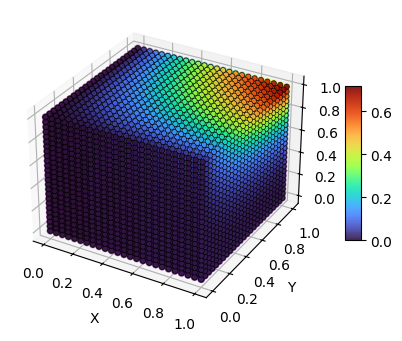}
      \caption*{$u_1$}
    \end{subfigure}
    \hspace{3mm}
    \begin{subfigure}{0.25\textwidth}
      \includegraphics[width=\linewidth]{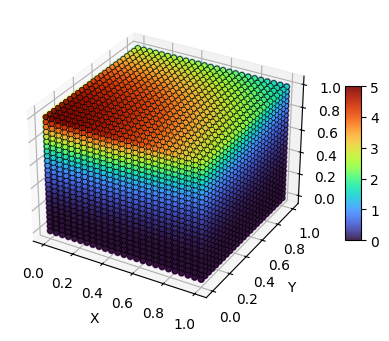}
      \caption*{$u_2$}
    \end{subfigure}
    \hspace{3mm}
    \begin{subfigure}{0.25\textwidth}
      \includegraphics[width=\linewidth]{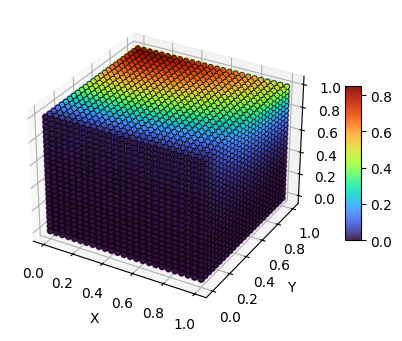}
      \caption*{$u_3$}
    \end{subfigure}
    \caption*{(a) True solution}
  \end{minipage}

  \begin{minipage}{\textwidth}
    \centering
    \begin{subfigure}{0.25\textwidth}
      \includegraphics[width=\linewidth]{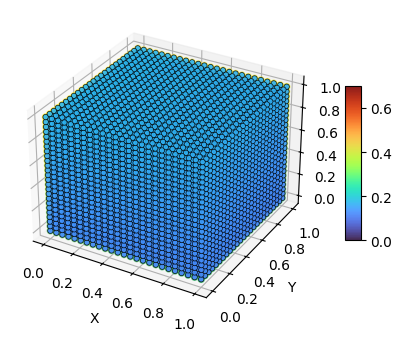}
      \caption*{$u_1$}
    \end{subfigure}
    \hspace{3mm}
    \begin{subfigure}{0.25\textwidth}
      \includegraphics[width=\linewidth]{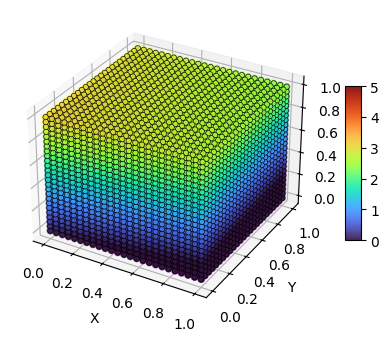}
      \caption*{$u_2$}
    \end{subfigure}
    \hspace{3mm}
    \begin{subfigure}{0.25\textwidth}
      \includegraphics[width=\linewidth]{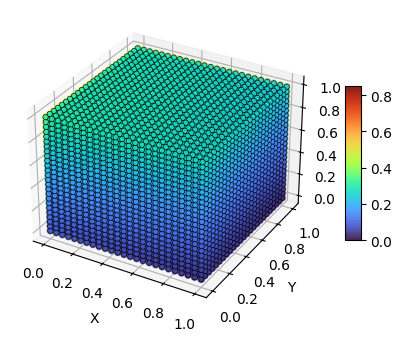}
      \caption*{$u_3$}
    \end{subfigure}
    \caption*{(b) Standard PINN}
  \end{minipage}

  \begin{minipage}{\textwidth}
    \centering
    \begin{subfigure}{0.25\textwidth}
      \includegraphics[width=\linewidth]{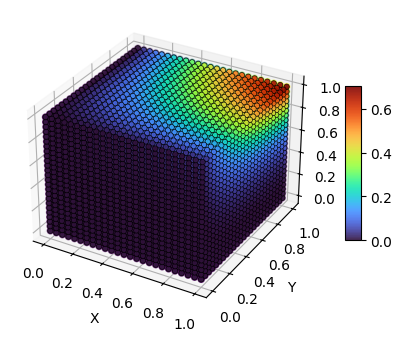}
      \caption*{$u_1$}
    \end{subfigure}
    \hspace{3mm}
    \begin{subfigure}{0.25\textwidth}
      \includegraphics[width=\linewidth]{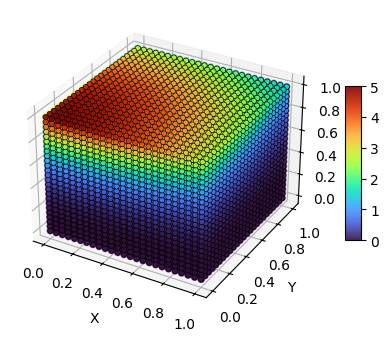}
      \caption*{$u_2$}
    \end{subfigure}
    \hspace{3mm}
    \begin{subfigure}{0.25\textwidth}
      \includegraphics[width=\linewidth]{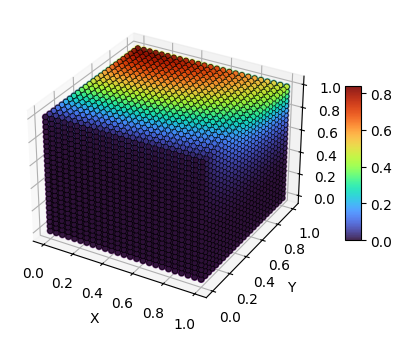}
      \caption*{$u_3$}
    \end{subfigure}
    \caption*{(c) Our method}
  \end{minipage}

  \caption{A comparison of solutions in three dimensions when $\lambda=10^4$: the exact solution (a), the solution computed by the standard PINN (b), and the solution obtained by our method (c). The experiment shows that standard PINNs fail to approximate the true solution in this example.}
  \label{ex2_sol_pred}
\end{figure}

\begin{figure}
    \centering
    \begin{subfigure}{0.30\textwidth}
        \centering
        \includegraphics[width=\linewidth]{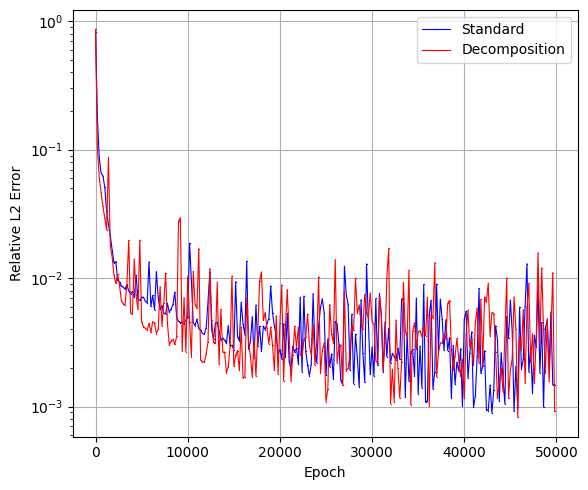}
        \caption{$\lambda=10$}
    \end{subfigure}   
    \begin{subfigure}{0.30\textwidth}
        \centering
        \includegraphics[width=\linewidth]{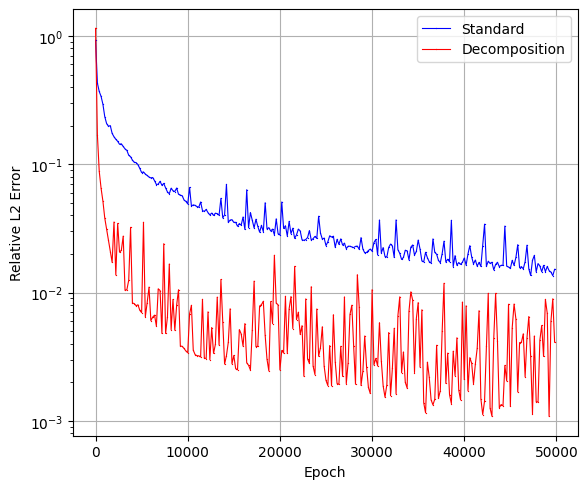}
        \caption{$\lambda=10^2$}
    \end{subfigure}    
    \begin{subfigure}{0.30\textwidth}
        \centering
        \includegraphics[width=\linewidth]{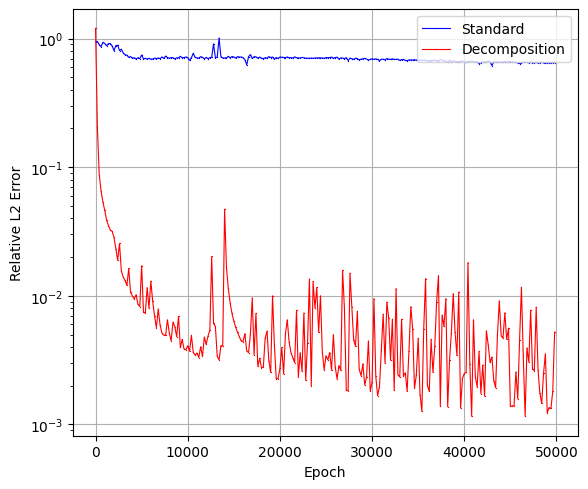} 
        \caption{$\lambda=10^4$}
    \end{subfigure}
    \caption{A comparison of learning curves with respect to the relative $L^2$-error between the standard PINN and our method for various values of $\lambda$ in three dimensions. The experiment shows that for $\lambda = 10^4$, the standard PINN approach does not yield an approximation to the true solution in this example.}
    \label{ex2_learn_curve}
\end{figure}

\subsection{Example 2 (Three-dimensional case)}
This example is devoted to illustrating that our method also performs well for a three-dimensional case. We set $\Omega=(0,1)^3$ and  $\mu=1$ in \eqref{eq:EE}-\eqref{eq:sig}. We choose  $\vf$  so that 
\[\vu(\boldsymbol{x})=\begin{bmatrix}  u_1(x,y,z)\\  u_2(x,y,z)\\u_3(x,y,z)\end{bmatrix}
=\begin{bmatrix} z^3\sin x\sin y\\ 5z^3\cos x \cos y\\z^4\cos x \sin y\end{bmatrix} +\frac{1}
{\lambda}\begin{bmatrix}\sin x\\ \sin y \\ \sin z \end{bmatrix},\] 
where $\boldsymbol{x}=(x,y,z) \in \Omega$. As before, we utilize a neural network with four hidden layers. Each layer consists of 128 neurons, and the GELU function is used as the activation function. To compute the solution, we aim to minimize the proposed loss function \eqref{main_loss} with coefficients  $\delta_r= 0.05$, $\delta_s= 1$, and $\delta_b= 300$, and numbers of random samples $N_r=N_s=15625$ and $N_b=1536$, and once again we adopt the Adam optimizer.

As in the two-dimensional case,  Figure \ref{ex2_sol_pred} presents a comparison between the solutions predicted by the standard PINN and those obtained using our proposed method when $\lambda=10^4$. As depicted in the figure,  the standard PINN produces a solution that is nearly zero with a $6.95$e$-1$ relative $L^2$-error, indicating its failure to accurately approximate the solution of the given problem. On the other hand, our proposed method yields solutions that closely match the exact solution, achieving the $9.85$e$-3$ relative $L^2$ error. Figure \ref{ex2_learn_curve} shows the comparison of learning curves for both the standard PINN and our method for different values of $\lambda$. As before, the standard PINN fails to accurately compute the solution as $\lambda$ increases, whereas our method consistently generates accurate numerical solutions. These experiments indicate that, particularly in the nearly incompressible case, the decomposition-based approach significantly improves the solution accuracy compared to the standard PINN, and the proposed method is robust against the value of $\lambda$ also in a three-dimensional domain.

\begin{figure}
  \centering
  \begin{minipage}{\textwidth}
    \centering
    \begin{subfigure}{0.25\textwidth}
      \includegraphics[width=\linewidth]{u1_sol.png}
      \caption*{$u_1$}
    \end{subfigure}
    \hspace{3mm}
    \begin{subfigure}{0.25\textwidth}
      \includegraphics[width=\linewidth]{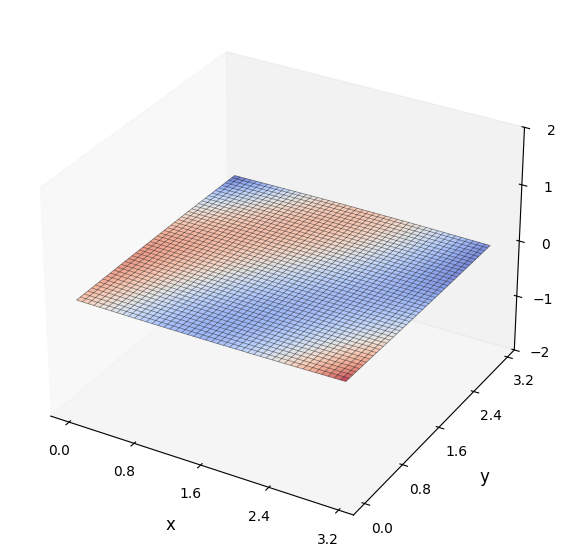}
      \caption*{$u_1$}
    \end{subfigure}
    \hspace{3mm}
    \begin{subfigure}{0.25\textwidth}
      \includegraphics[width=\linewidth]{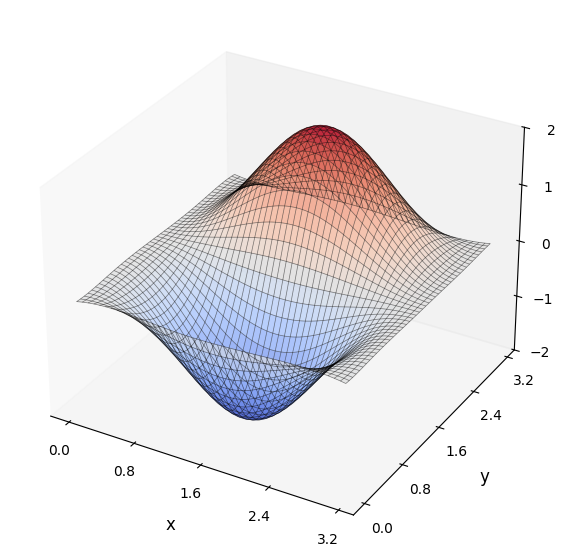}
      \caption*{$u_1$}
    \end{subfigure}
  \end{minipage}

  \begin{minipage}{\textwidth}
    \centering
    \begin{subfigure}{0.25\textwidth}
      \includegraphics[width=\linewidth]{u2_sol.png}
      \caption*{$u_2$ \\ \vspace{2mm} (a) True solution}
    \end{subfigure}
    \hspace{3mm}
    \begin{subfigure}{0.25\textwidth}
      \includegraphics[width=\linewidth]{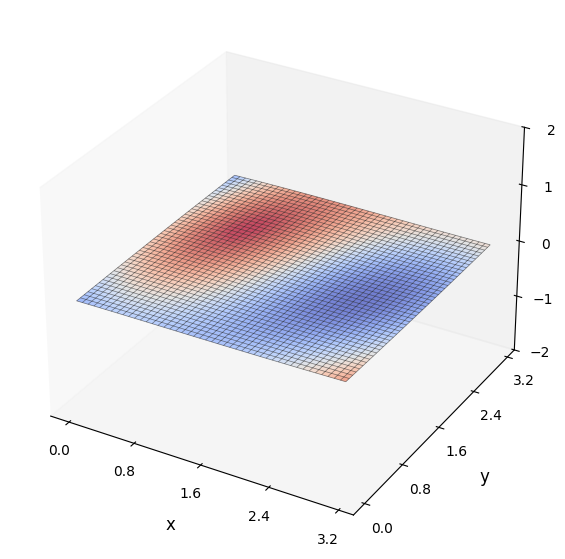}
      \caption*{$u_2$ \\  \vspace{2mm} (b) Standard PINN}
    \end{subfigure}
    \hspace{3mm}
    \begin{subfigure}{0.25\textwidth}
      \includegraphics[width=\linewidth]{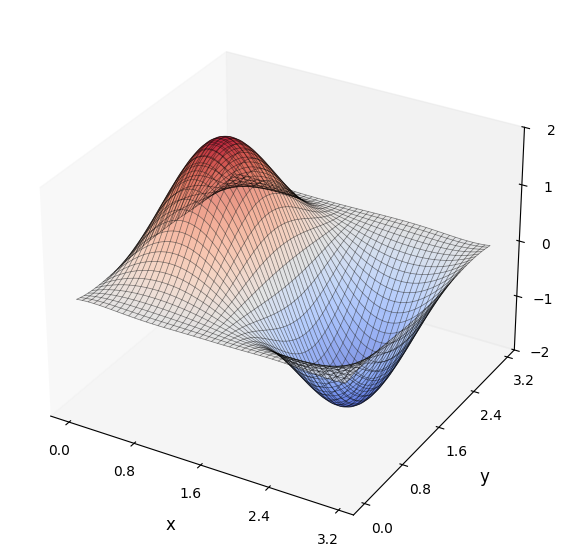}
      \caption*{$u_2$ \\  \vspace{2mm} (c) Our method}
    \end{subfigure}
  \end{minipage}
  \caption{A comparison of solutions for the case of variable Lam\'e coefficients with $\Lambda=10^4$: the exact solution (a), the solution computed by the standard PINN (b), and the solution obtained by our method (c). The experiment shows that standard PINNs fail to approximate the true solution in this example.}
  \label{ex4_sol_pred}
\end{figure}

\begin{figure}
    \centering
    \begin{subfigure}{0.30\textwidth}
        \centering
        \includegraphics[width=\linewidth]{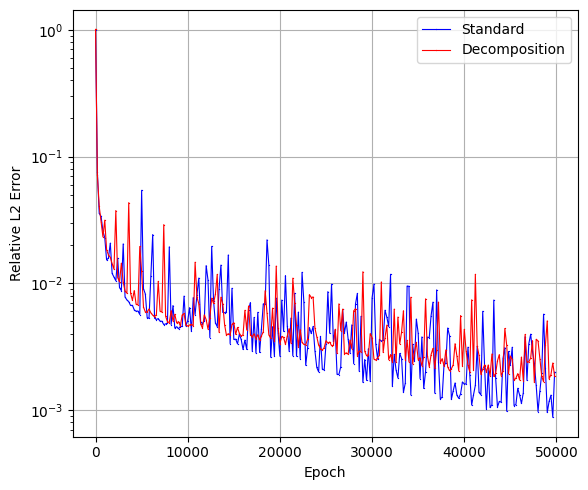}
        \caption{$\Lambda=10$}
    \end{subfigure}   
    \begin{subfigure}{0.30\textwidth}
        \centering
        \includegraphics[width=\linewidth]{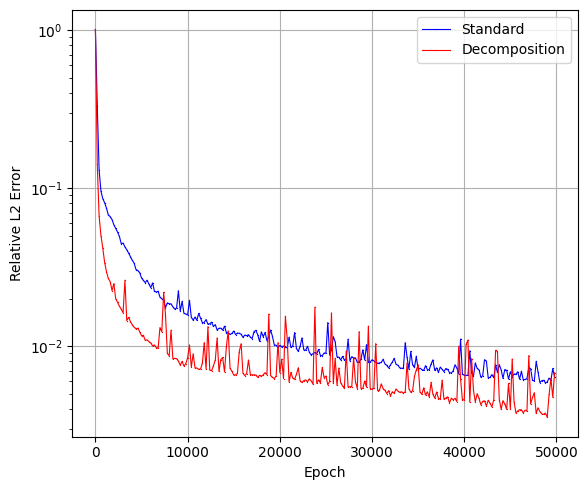}
        \caption{$\Lambda=10^2$}
    \end{subfigure}    
    \begin{subfigure}{0.30\textwidth}
        \centering
        \includegraphics[width=\linewidth]{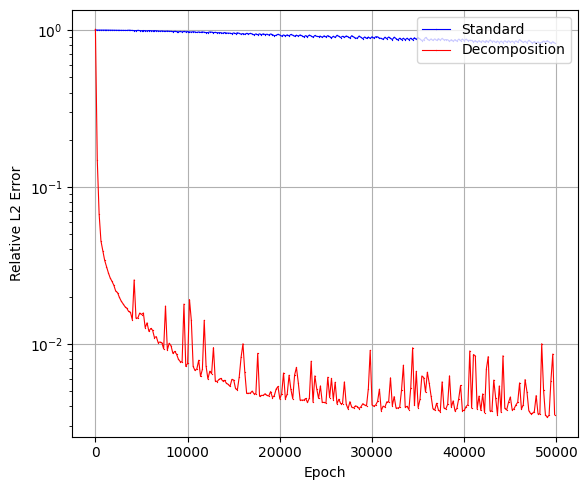} 
        \caption{$\Lambda=10^4$}
    \end{subfigure}
    \caption{A comparison of learning curves with respect to the relative $L^2$-error between the standard PINN and our method for the case of variable Lam\'e coefficients with various values of $\Lambda$. The experiment shows that for $\lambda = 10^4$, the standard PINN approach does not yield an approximation to the true solution in this example.}
    \label{ex4_learn_curve}
\end{figure}

\subsection{Example 3 (Inhomogeneous material)}
In this example, we consider the elasticity equation \eqref{eq:EE} with variable  Lam\'e coefficients $\mu$ and $\lambda$. For $\bsx=(x,y)\in (0,\pi)^2$,  we set $\lambda(\boldsymbol{x}) = \Lambda(1+\tfrac12\sin(2x))$  and $\mu(\boldsymbol{x}) = 1 + x+y$. Here, \(\Lambda\) is a positive scaling parameter that tends to infinity in the nearly incompressible limit, reflecting the material’s increasing resistance to volumetric change. We choose  the body force $\vf$  so that 
\[\vu(\boldsymbol{x})=\begin{bmatrix}  u_1(x,y)\\  u_2(x,y)\end{bmatrix}
=\begin{bmatrix} \bigl(\cos(2x)-1\bigr)\sin(2y)\\ \bigl(1-\cos(2y)\bigr)\sin(2x)
\end{bmatrix} +\frac{\sin(x)\sin(y)}{\Lambda}\begin{bmatrix}1\\ 1\end{bmatrix}.\] 

For this experiment, we consider a neural network architecture composed of four hidden layers, each consisting of 64 neurons, with GELU activation functions. The training procedure employs the decomposed loss function \eqref{main_loss}, with weighting parameters set as $\delta_r= 0.05$, $\delta_s= 1$ and $\delta_b= 20$, along with sampling sizes $N_r = 5000$, $N_s = 5000$, and $N_b = 400$, respectively. The optimization procedure is carried out using the Adam optimizer.

Figure \ref{ex4_sol_pred} compares the predicted solutions obtained by the standard PINN and the proposed method for $\Lambda = 10^4$. It can be clearly observed from the figure that the standard PINN produces solutions very close to zero, yielding a relative $L^2$-error of $9.95$e$-1$, thus failing to predict the solution accurately. In contrast, our proposed method provides a significantly improved prediction closely matching the exact solution, with a relative $L^2$-error reduced to $8.64$e$-3$. To further evaluate the effectiveness of the proposed approach across different material parameters, we compared learning curves for varying values of $\Lambda$, computing the relative $L^2$-error at each epoch. These comparisons are summarized in Figure \ref{ex4_learn_curve}. As illustrated in Figure \ref{ex4_learn_curve} (a), when $\Lambda = 10$, indicating negligible locking effects, the training performances of the standard PINN and our proposed approach are comparable. However, as $\Lambda$ increases, as shown in Figures \ref{ex4_learn_curve} (b) and (c), the proposed method exhibits distinctly superior performance in terms of accuracy. Notably, Figure \ref{ex4_learn_curve} (c) (when $\Lambda=10^4$, the inhomogeneous material becomes nearly incompressible) highlights that the standard PINN fails to converge, whereas the proposed approach maintains effective training and generates consistently accurate predictions.

\subsection{Example 4 (Parametric problem)}

As discussed in Section \ref{sec:para_PINN}, one of the principal advantages of neural network-based approaches lies in their ability to efficiently approximate solutions for parametric PDE problems. Specifically, the parametric PINN framework extends the original PINN formulation by incorporating an integration of the loss function over the parameter space. 

In the practical application of elasticity, both the Young's modulus $E$ and the Poisson's ratio $\nu$ are typically obtained from experimental measurements of the material being modeled. Owing to this, it is natural to 
express  $\lambda$ and $\mu$ in terms of $E$ and  $\nu$ via the following relation: 
\[
\lambda=\frac{E\nu}{(1+\nu)(1-2\nu)}\quad\text{and}\quad
\mu= \frac{E}{2(1+\nu)},
\]
and train a model providing rapid solution predictions corresponding to varying values of $\mu$ and $\lambda$. In other words, the aim in this example is to look for a parametric solution $\vu$ as a function of  $E$ and  $\nu$. In this formulation, the parameters $\boldsymbol{p}=(E, \nu)$ are sampled from the parametric domain $\mathcal{P}=(2,4) \times (0.1,0.5)$, which encompasses, in particular, the near-incompressibility regime.

For the experiment, we employ the same setting as used in Example 1, where we consider the linear elasticity equation \eqref{eq:EE} with the true solution given in \eqref{ex1}.  
Utilizing the notation established in Section \ref{sec:para_PINN}, we choose $\delta_r=1$ and $\delta_b=20$ in \eqref{parametric_loss1}, that is, the parametric loss function $\mathcal{L}_{\rm para} = \mathcal{L}_{{\rm{para}},\,r} + 20\mathcal{L}_{{\rm{para}},\,b},$ where  $\mathcal{L}_{{\rm{para}},\,r}$ and $\mathcal{L}_{{\rm{para}},\,b}$ are defined in \eqref{parametric_loss}.  The neural network architecture consists of four hidden layers with 64 neurons per layer and employs GELU as an activation function. For the training procedure, we used the Adam optimizer and generated sampling points from a uniform distribution, with $N_{pb} \times N_b = 30{,}000$ and $N_{pr} \times N_r = 100{,}000$ used for the boundary and residual components, respectively.

As previously noted, a primary advantage of this approach is its capability to generate real-time predictions of the solution in response to variations in the parameters $E$ and $\nu$. This represents a significant improvement over the standard PINN approaches, which necessitate retraining the model whenever the parameters are varied. Figure \ref{para_error_surf} presents surface plots of the relative $L^2$-error for both the standard PINN and our proposed method, as the pair of parameters $\boldsymbol{p}=(E,\nu)$ varies within the specified range $\mathcal{P}=(2,4)\times(0.1,0.5)$. As observed in Figure \ref{para_error_surf}, our proposed method yields accurate solutions across the entire parameter domain. In particular, while the standard PINN exhibits a sharp increase in error as $\nu$ approaches $\frac{1}{2}$, corresponding to the nearly incompressible regime, our method maintains stable and accurate predictions even in this region, known to cause significant computational challenges. This demonstrates the robustness of our approach with respect to variations in the material parameters.

\begin{figure}
    \centering
    \begin{subfigure}{0.4\textwidth}
        \centering
        \includegraphics[width=\linewidth]{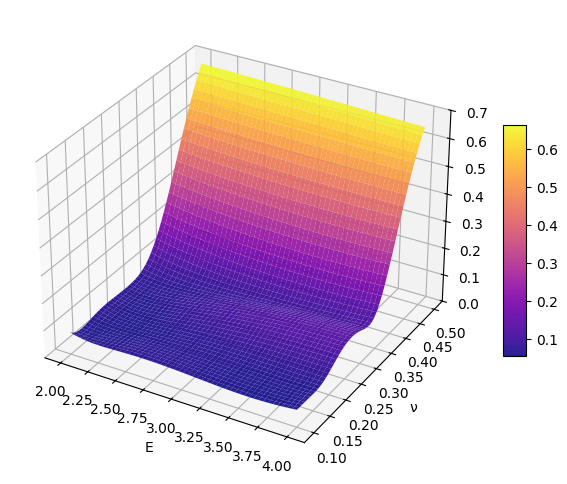}
        \caption{Standard PINN}
    \end{subfigure}   
    \begin{subfigure}{0.4\textwidth}
        \centering
        \includegraphics[width=\linewidth]{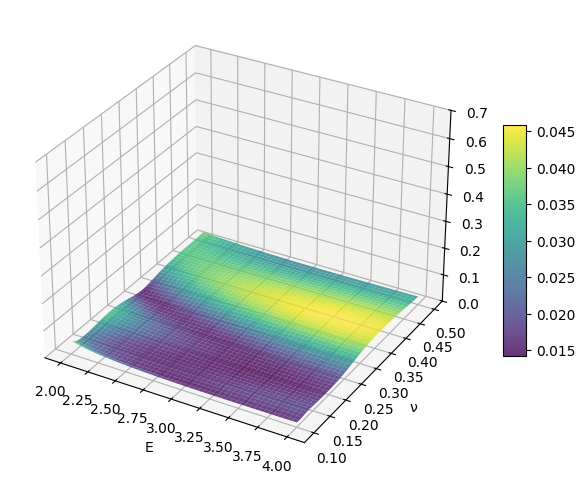}
        \caption{Our method}
    \end{subfigure}   
    \caption{Error surfaces for the relative $L^2$-errors in the parametric solution $\vu(\cdot;E,\nu)$ across varying values of $(E,\nu)$. The experiment shows that for standard PINN, there is a sharp increase in the error as the material becomes nearly incompressible (i.e., $\nu$ approaches $1/2$).}
    \label{para_error_surf}
\end{figure}

To further demonstrate the effectiveness of the proposed method in mitigating the locking phenomenon, we evaluated the error with respect to varying values of $\nu$ while keeping $E$ fixed. Specifically, we divided the range of $\nu$ into seven intervals: $(0.1,0.2)$, $(0.2,0.3)$, $(0.3,0.4)$, $(0.4,0.45)$, $(0.45,0.49)$, $(0.49,0.499)$ and $(0.499,0.4999)$. For each interval, 100 random values of $\nu$ were sampled, and the averaged relative $L^2$-error was measured. As shown in the Figure \ref{Parametric_overall}, our method consistently produced accurate solutions across all intervals. In particular, even when the $\nu$ is close to $0.5,$ where locking typically occurs, the proposed method produces accurate and stable solution predictions.

\begin{figure}
    \centering    \includegraphics[width=0.5\textwidth]{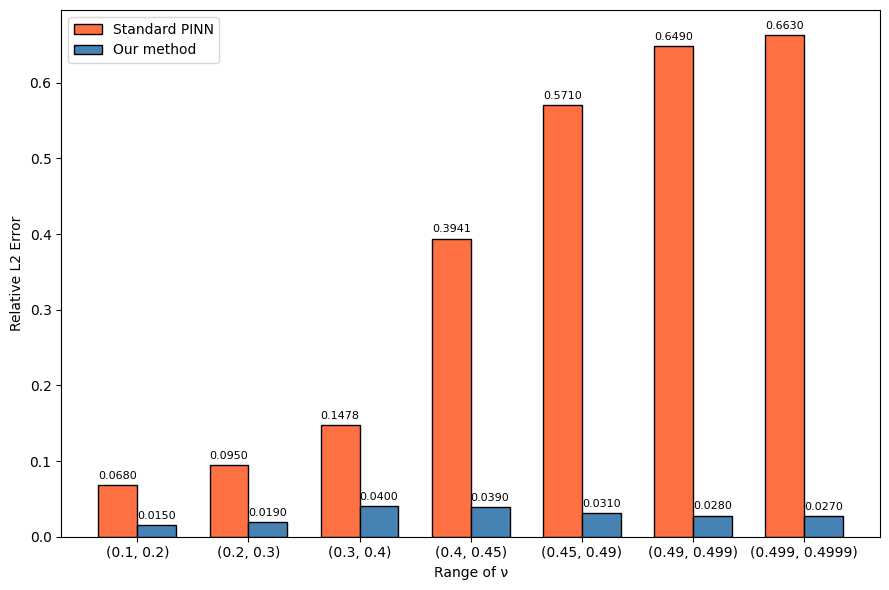}
    \caption{Comparison of the average relative $L^2$-errors in the parametric solution $\vu(\cdot;E,\nu)$ of the standard PINN and our method over $100$ samples in each interval with different range of $\nu$. The experiment shows that a standard PINN approach is sensitive to the value of $\nu$, especially in the nearly incompressible regime (i.e., $\nu$ approaches $1/2$).}
    \label{Parametric_overall}
\end{figure}

\section{Conclusion}\label{sec_conclusion}
In this paper, we investigate the linear elasticity equations within the nearly incompressible regime and propose a novel and robust approach utilizing PINNs. Unlike the conforming FEM, where locking (non-robustness) phenomena occur primarily due to approximation limitations, locking in PINNs arises predominantly from a pronounced imbalance between the Lam\'e coefficients, thereby causing significant difficulties in optimization. To resolve this issue, we introduce a decomposition of the governing equations into two carefully structured subsystems, separating the effect of each parameter. Our proposed approach concurrently addresses forward and inverse problems during the training, automatically identifying external forces and boundary conditions suitable for the considered scenarios. The efficiency and robustness of our methodology are validated through comprehensive numerical experiments, including both two- and three-dimensional examples, problems involving variable coefficients, and parametric settings. The results consistently demonstrate the capability of our method to effectively handle elastic materials that become nearly incompressible, highlighting its potential applicability to complex elasticity problems in computational mechanics.

An application of our method to more practical problems in real-world scenarios would form an interesting future research direction. Moreover, the decomposition-based approach we propose is expected to be applicable to a broader class of problems in which the quality of numerical approximation is highly sensitive to parameter values, such as singular perturbation problems involving small diffusion parameters, or the Navier--Stokes equations with large Reynolds numbers. These possibilities will be explored in future works.

\section*{Acknowledgments}
Seungchan Ko and Sanghyeon Park are supported by National Research Foundation of Korea Grant funded by the Korean Government (RS-2023-00212227). Josef Dick, Quoc Thong Le Gia, and Kassem Mustapha are partially supported by the ARC Grant DP220101811.


\begin{thebibliography}{99.}
\bibitem{AlmeidaSilvaJr2023} E. F. de Almeida, S. da Silva, and A. Cunha Jr, Physics-informed neural networks for solving elasticity problems, Proceedings of the 27th International Congress on Mechanical Engineering (COBEM 2023) (2023).

\bibitem{AinsworthParker2022} M. Ainsworth and C. Parker, Unlocking the secrets of locking, Comput. Methods Appl. Mech. Engrg., 395, 115034 (2022).

\bibitem{ArnoldAwanouWinther2014} D. N. Arnold, G. Awanou and  R. Winther, Nonconforming tetrahedral mixed finite elements for elasticity, Math. Models Methods Appl. Sci., 24, 783--796 (2014).

\bibitem{BabuskaSuri1992}  I. Babuška and M. Suri, Locking effects in the finite element approximation of elasticity problems, Numer. Math., 62, 439--463 (1992).


\bibitem{auto_diff_1} A. G. Baydin, B. A. Pearlmutter, A. A. Radul and J. M. Siskind, Automatic Differentiation in Machine Learning: a Survey, J. Mach. Learn. Res., 18, 5595--5637 (2017).


\bibitem{BlechschmidtErnst2021} J. Blechschmidt and O. G. Ernst, Three ways to solve partial differential equations with neural networks – A review,  GAMM-Mitteilungen, 44, e202100,006 (2021). 

\bibitem{BramwellDemkowiczGopalakrishnanQiu2012} J. Bramwell, L. Demkowicz, J. Gopalakrishnan and  W. Qiu, A locking-free hp DPG method for linear elasticity with symmetric stresses, Numer. Math., 122, 671--707 (2012). 

\bibitem{BrennerSung1992} S. C. Brenner and  L.-Y. Sung, Linear finite element methods for planar linear elasticity, Math. Comput., 59, 321--338 (1992).

\bibitem{CaiMaoWang2021} S. Cai, Z. Mao, Z. Wang, M. Yin, and G. E. Karniadakis, Physics-informed neural networks (PINNs) for fluid mechanics: a review, Acta. Mech. Sin., 37, 1727--1738 (2021). 

\bibitem{ChenGu2023} C.-T. Chen and G. X. Gu, Physics-informed deep-learning for elasticity: forward, inverse, and mixed problems, Adv. Sci., 10, 2300439 (2023).

\bibitem{ChenXie2016} G. Chen and  X. Xie, A robust weak Galerkin finite element method for linear elasticity with strong symmetric stresses, Comput. Methods Appl. Math., 16, 389--408 (2016).


\bibitem{Cuomoetal2022} S. Cuomo, V. Schiano di Cola, F. Giampaolo, G. Rozza, M. Raissi, and F. Picciali, Scientific machine learning through physics–informed neural networks: Where we are and what's next, Journal of Scientific Computing, 92, 88, (2022).
 
\bibitem{VeigaBrezziMarini2013} L. B. Da Veiga, F. Brezzi, and  L. D. Marini, Virtual elements for linear elasticity problems, SIAM J. Numer. Anal., 51, 794--812 (2013).



\bibitem{DiPietroNicaise2013} D. A.  Di Pietro and S. Nicaise, A locking-free discontinuous Galerkin method for linear elasticity in locally nearly incompressible heterogeneous media, Appl. Numer. Math., 63, 105--116 (2013).

\bibitem{EdoardoStefanoCarloLuca2020} A. Edoardo, M. Stefano, L. Carlo and P. Luca,  A dual hybrid virtual element method for plane elasticity problems, ESAIM: Math. Model. Numer. Anal., 54, 1725--1750 (2020).

\bibitem{EskinDavydovGurevaMalkhanovSmorkalov2024} V. A. Es'kin, D. V. Davydov, J. V. Gur'eva, A. O. Malkhanov, and M. E. Smorkalov, Separable physics-informed neural networks for the solution of elasticity problems (2024), arxiv.org/abs/2401.13486.

\bibitem{Falk1991} R. Falk, Nonconforming finite element Methods for the equations of linear elasticity, Math. Comput., 57, 529--550 (1991).

\bibitem{Trudinger} 
D. Gilbarg, and N. S. Trudinger, 
Elliptic partial differential equations of second order, Springer-Verlag, Berlin (2001).

\bibitem{GopalakrishnanGuzman2011} J. Gopalakrishnan and J. Guzm\'an, Symmetric nonconforming mixed finite elements for linear elasticity, SIAM J. Numer. Anal. 49, 1504--1520 (2011).

\bibitem{GuoHaghighat2022} M. Guo and E. Haghighat, Energy-based error bound of physics-informed neural network solutions in elasticity,  J. Engineering Mechanics, 148, 04022038 (2022).

\bibitem{HaghighatRaissiMoure2021} E. Haghighat, M. Raissi, A. Moure, H. Gomez, and R. Juanes, A physics-informed deep learning framework for inversion and surrogate modeling in solid mechanics, Comput. Methods Appl. Mech. Engrg., 379, 113, 741 (2021b).

\bibitem{HansboLarson2002} P. Hansbo and  M. G. Larson, Discontinuous Galerkin methods for incompressible and nearly incompressible elasticity by Nitsche’s method, Comput. Methods Appl. Mech. Eng., 191, 1895--1908 (2002).

\bibitem{HuoWangWangZhang2020}  F. Huo, R. Wang, Y. Wang and R. Zhang, A locking-free weak Galerkin finite element method for linear elasticity problems,  Computers Math. Appl., 160, 181--190 (2024).

\bibitem{pPINN} A. Kaltenbach and M. Zeinhofer, The deep Ritz method for parametric $p$-dirichlet
problems (2022), arxiv.org/pdf/2207.01894v1.

\bibitem{KarniadakisKevrekidisLu2021} G. E. Karniadakis, Y. Kevrekidis, and L. Lu, P. Perdikaris, Physics-informed machine learning, Nature Reviews Phys., 3, 422--440 (2021).

 \bibitem{KharazmiZhangKarniadakis2021} E. Kharazmi, Z. Zhang, and G. E. M. Karniadakis, hp-VPINNs: Variational physics-informed neural networks with domain decomposition, Comput. Methods Appl. Mech. Engrg., 374, 113, 547 (2021b).


\bibitem{vs_pinn} S. Ko and S. Park, VS-PINN: A fast and efficient training of physics-informed neural networks using variable-scaling methods for solving PDEs with stiff behavior,   J. Comput. Phys.,  529, 113860 (2025).

\bibitem{LeeLeeSheen2003} C.-O. Lee, J. Lee and  D. Sheen, A locking-free nonconforming finite element method for planar linear elasticity, Adv. Comput. Math., 19, 277--291 (2003).

\bibitem{LiuWang2022} Y. Liu and  J. Wang, A locking-free $P_0$ finite element method for linear elasticity equations on polytopal partitions, IMA J. Numer. Anal., 42, 3464--3498 (2022).

\bibitem{MaoChen2008} S. Mao and  S. Chen, A quadrilateral nonconforming finite element for linear elasticity problem, Adv. Comput. Math., 28, 81--100 (2008).

\bibitem{Mustaphaetal2024} K. Mustapha, W. McLean, J. Dick, and Q. T. Le Gia,  A simple modification to mitigate locking in conforming FEM for nearly incompressible elasticity (2024) https://arxiv.org/html/2407.06831v3.

\bibitem{auto_diff_2} A. Paszke, S. Gross, S. Chintala, G. Chanan, E. Yang, Z. DeVito, Z. Lin, A. Desmaison, L. Antiga and A. Lerer,
Automatic differentiation in PyTorch (2017).

\bibitem{RaissiPerdikarisKarniadakis2017c} M. Raissi, P. Perdikaris, and G. E.  Karniadakis, Physics Informed Deep Learning (Part I): Data-driven solutions of nonlinear partial differential equations  (2017c) arXiv:1711.10561. 


\bibitem{RaissiPerdikarisKarniadakis2017d} M. Raissi, P. Perdikaris, and G. E.  Karniadakis, Physics Informed Deep Learning (Part II): Data-driven discovery of nonlinear partial differential equations (2017d) arXiv:1711.10566.


\bibitem{RaissiPerdikarisKarniadakis2019} M. Raissi, P. Perdikaris, and G. E.  Karniadakis, Physics-informed neural networks: A deep learning framework for solving forward and inverse problems involving nonlinear partial differential equations, J. Comput. Phys., 378, 686--707 (2019). 


\bibitem{RoyBoseSundararaghavanArroyave2023} A. M. Roy, R. Bose, V. Sundararaghavan and R. Arr\'oyave, Deep learning-accelerated computational framework based on physics-informed neural network for the solution of linear elasticity, Neural Networks, 162, 472--489 (2023).


\bibitem{ScottVogelius1985} L. R. Scott and  M. Vogelius, Norm estimates for a maximal right inverse of the divergence operator in spaces of piecewise polynomials, RAIRO Math. Model. Num. Anal., 19, 111--143 (1985).

\bibitem{SoonCockburnStolarski2009} S.-C. Soon, B. Cockburn and H. K. Stolarski, A hybridizable discontinuous Galerkin method for linear elasticity, Int. J. Numer. Meth.  Eng., 80, 1058--1092 (2009).

\bibitem{SunGaoPan2020a} L. Sun, H. Gao, S. Pan, and J.-X. Wang, Surrogate modeling for fluid flows based on physics-constrained deep learning without simulation data, Comput. Methods Appl. Mech. Engrg., 361, 112, 732 (2020a).  


\bibitem{ZhangZhaoYangChen2019} B. Zhang, J. Zhao, Y. Yang, and  S. Chen, The nonconforming virtual element method for elasticity problems, J. Comput. Phys., 378, 394--410 (2019).

\bibitem{ZhuZabarasKoutsourelakis2019} Y. Zhu, N. Zabaras, P.-S. Koutsourelakis, and P. Perdikaris, Physics-constrained deep learning for high-dimensional surrogate modeling and uncertainty quantification without labeled data,   J. Comput. Phys.,  394, 56--81 (2019).

\bibitem{ZeinhoferMasriMardal2024} 
M.~Zeinhofer, R.~Masri, K.-A. Mardal,
A Unified Framework for the Error Analysis of Physics-Informed Neural Networks, IMA J. Num. Anal. (2024).

\end{thebibliography}
\end{document}